\documentclass{amsart}

\usepackage[utf8x]{inputenc}
\usepackage[T2A]{fontenc}

\usepackage{amsmath,amsthm,amsopn,amstext,amscd,amsfonts,amssymb,mathrsfs,mathtools}
\usepackage{url}
\usepackage{booktabs}
\usepackage{hyperref}
\usepackage{dsfont}
\hypersetup{
	colorlinks   = true,
	citecolor    = blue,
	linkcolor = blue
}
\usepackage{comment}
\usepackage{xcolor}
\usepackage{float}
\usepackage{multirow}
\usepackage[active]{srcltx}
\usepackage{graphicx, mathdots}

  \usepackage{mathtools}

\newtheorem{problem}{\sc Problem}
\newtheorem{definition}{\sc Definition}[section]
\newtheorem{theorem}{\sc Theorem}[section]

\newtheorem{lemma}{\sc Lemma}[section]

\newtheorem{coro}{\sc Corollary}[section]
\newtheorem{prop}{\sc Proposition}[section]

\usepackage[active]{srcltx}

\usepackage{graphicx, epsfig, subfig}
\usepackage{lscape}
\usepackage{rotating}
\usepackage{caption}
\usepackage{multicol}
\usepackage{colortbl}
\usepackage{nicematrix}
\usepackage{enumitem}

\ExplSyntaxOn

\NewDocumentCommand{\pFq}{O{}mmmmm}
 {
  \group_begin:
  \keys_set:nn { hypergeometric } { #1 }
  \hypergeometric_print:nnnnn { #2 } { #3 } { #4 } { #5 } { #6 }
  \group_end:
 }
\NewDocumentCommand{\hypergeometricsetup}{m}
 {
  \keys_set:nn { hypergeometric } { #1 }
 }

\tl_new:N \l_hypergeometric_divider_tl
\tl_new:N \l_hypergeometric_left_tl
\tl_new:N \l_hypergeometric_right_tl

\keys_define:nn { hypergeometric }
 {
  symbol .tl_set:N = \l_hypergeometric_symbol_tl,
  symbol .initial:n = F,
  separator .tl_set:N = \l_hypergeometric_separator_tl,
  separator .initial:n = {},
  skip .tl_set:N = \l_hypergeometric_skip_tl,
  skip .initial:n = 8,
  divider .choice:,
  divider/semicolon .code:n = \tl_set:Nn \l_hypergeometric_divider_tl { \;; },
  divider/bar .code:n = \tl_set:Nn \l_hypergeometric_divider_tl { \;\middle|\; },
  divider .initial:n = semicolon,
  fences .choice:,
  fences/brack .code:n = 
   \tl_set:Nn \l_hypergeometric_left_tl {[}
   \tl_set:Nn \l_hypergeometric_right_tl {]},
  fences/parens .code:n = 
   \tl_set:Nn \l_hypergeometric_left_tl {(}
   \tl_set:Nn \l_hypergeometric_right_tl {)},
  fences .initial:n = brack,
 }

\cs_new_protected:Nn \hypergeometric_print:nnnnn
 {
  {} \sb {#1} \l_hypergeometric_symbol_tl \sb { #2 }
  \left\l_hypergeometric_left_tl
  \genfrac .. 
           {0pt} 
           {} 
           { \__hypergeometric_process:n { #3 } } 
           { \__hypergeometric_process:n { #4 } } 
  \l_hypergeometric_divider_tl
  #5
  \right\l_hypergeometric_right_tl
 }

\cs_new_protected:Nn \__hypergeometric_process:n
 {
  \clist_use:nn { #1 }
   {
    {\l_hypergeometric_separator_tl}
    \mspace { \l_hypergeometric_skip_tl mu }
   }
 }

\ExplSyntaxOff

\hypergeometricsetup{
    symbol=F,
  fences=brack,
  separator={,},
  divider=semicolon,
}

\newtheoremstyle{custom}%
  {}                
  {}                
  {\itshape}        
  {}                
  {\bfseries}       
  {}                
  {.5em}            
  {}                

\theoremstyle{custom}

\newenvironment{custom}[1]{%
    \noindent\textbf{$\star$ #1}.
}{%
    \par 
}

\begin{document}
\title[The Nikiforov-Uvarov method]{The Nikiforov-Uvarov method}
\author{G. Gordillo-N\'u\~nez}
\address{CMUC, Department of Mathematics, University of Coimbra, Portugal}
\email{up202310693@edu.fc.up.pt}
\subjclass[2010]{}
\date{\today}
\keywords{}
\begin{abstract}
We review the so-called Nikiforov-Uvarov method along with some basic results about classical orthogonal polynomials and hypergeometric functions related to the hypergeometric differential equation. The method is employed to address certain eigenvalue problems that appear in quantum mechanics, namely, time-independent Schr\"odinger equation, paying special attention to properly and completely characterising its spectrum. Finally, some potentials are discussed and solved.
\end{abstract}
\maketitle
\section{Introduction}
As stated on their influential book {\it Special Functions of Mathematical Physics: A Unified Introduction with Applications}, see \cite{Nikiforov1988SpecialFO}, by Russian mathematicians and physicists Arnold F. Nikiforov and Vasilii B. Uvarov (\cite[p. 1]{Nikiforov1988SpecialFO}):
\vspace{2mm}
\begin{changemargin}{0.8cm}{0.8cm} 
{\em Many important problems in theoretical and mathematical physics lead to the $[$so-called generalized hypergeometric or GHE \!$]$ differential equation
$$
u''+\dfrac{\widetilde{\psi}(z)}{\phi(z)}u'+\dfrac{\widetilde{\phi}(z)}{\phi^2(z)}u=0,
$$
where $\phi(z)$ and $\widetilde{\phi}(z)$ are polynomials of degree at most 2, and $\widetilde{\psi}(z)$ is a polynomial of degree at most 1. Equations of this form arise, for example, $[...]$ in the discussion of such fundamental problems of quantum mechanics as the motion of a particle in a spherically symmetric field, the harmonic oscillator, the solution of the Schr\"odinger, Dirac and Klein-Gordon equations for a Coulomb potential, and $[...]$ in typical problems of atomic, molecular and nuclear physics.}
\end{changemargin}
\vspace{2mm}
To solve such differential equation, they elegantly reduce it to a simpler hypergeometric differential equation or HDE of the form 
$$
\phi(z)y''+\psi(z)y'+\lambda y = 0,
$$
where $\psi(z)$ and $\phi(z)$ are polynomials of degrees at most 1 and 2, respectively, and $\lambda$ is a constant;
by taking $u=\chi(z)y$ and choosing an appropriate $\chi(z)$ obtained in a simple manner. Subsequently, they completely solve the hypergeometric differential equation and, among its solutions, classical orthogonal polynomials, in particular, and hypergeometric functions, more generally, are encountered.  

The aim of the present paper is to review such methodology in relation to solving the so-called time-independent Schr\"odinger equation (of capital importance in Quantum Mechanics) while taking special care that its spectrum is completely and properly characterised. In addition, we present three examples of exactly solvable systems via this so-called Nikiforov-Uvarov or NU method. The (time-independent) onedimensional Schr\"odinger equation or SE reads
\begin{equation}
\label{schro}
    -\dfrac{\hbar^2}{2m}\dfrac{\mathrm{d^2}\Psi}{\mathrm{d}x^2}(x)+V(x)\Psi(x)=E\Psi(x), \quad x \in \mathbb{R},
\end{equation}
where $\hbar$ is the reduced Planck constant, $m$ is the mass of the particle subjected to the potential $V(x)$, $\Psi(x)$ is its wave function and $E$ is its energy which acts as an eigenvalue or eigenenergy. Out of simplicity, adimensionalize Eq. \eqref{schro} and obtain
\begin{equation}
\dfrac{\mathrm{d^2}\Psi}{\mathrm{d}x^2}(x)+\left(\varepsilon-v(x)\right)\Psi(x)= 0, \quad x \in \mathbb{R},
\label{schroadi}
\end{equation}
were $\varepsilon$ and $v(x)$ are the reduced eigenenergy and potential, respectively. The problem, thus, consists of finding every pair $(\varepsilon, \Psi_{\varepsilon}(x))$ of {\em physical states} satisfying Eq. \eqref{schroadi}. To establish what we understand as a physical state, firstly, define
\begin{equation}
\label{v+-}
    v_{\text{min}} \coloneqq \min_{x \in \mathbb{R}}{v(x)}, \quad v_{\pm} \coloneqq \lim_{x \to \pm \infty} v(x), \quad v_{\text{max}} \coloneqq \max_{x \in \mathbb{R}}{v(x)},
\end{equation}
and, subsequently, the following two energy regions:
\begin{enumerate}[label=\text{R\arabic*}]
\item \label{R1} For $\varepsilon \in (v_{\text{min}},v_-)$, we distinguish the region of the so-called {\it bound states}, where the particle is classically confined in a finite region of space and the eigenvalue problem should have a non-trivial function, $\Psi(x)$, satisfying that $\Psi(x) \in L^{2}(\mathbb{R})$, this is, the space of square-integrable functions in $\mathbb{R}$.
\item \label{R2} For $\varepsilon  > v_-$, we distinguish the region of the so-called {\it scattering states}: 
the particle is able to reach $-\infty$ and/or $+\infty$ and square-integrability condition on its eigenfunction is dropped and substituted by $\Psi(x) \in L^{\infty}(\mathbb{R})$, this is, the space of bounded functions on $\mathbb{R}$.
\end{enumerate}
Notice that energy region $\varepsilon \in (-\infty, v_{\text{min}}]$ is excluded as, in a classical regime, particle are unable to propagate with such energy. So that {\em observability}\footnote{For a proper introduction to Quantum Mechanics including the concept of observability, reader is referred to \cite[Chapter 2]{Takhtajan2008QuantumMF}, for instance.} of the system is achieved, i.e, so that its associated states form a complete orthonormal system (COS) for $L^2(\mathbb{R})$, the problem of solving SE can be understood, in the adequate context, as an eigenvalue problem as follows:
\begin{problem}
\label{schroProb}
Find all non-trivial pairs $\left(\varepsilon,\Psi_{\varepsilon}(x)\right)$ with $\varepsilon > v_-$ satisfying Eq. \eqref{schroadi}
for which
\begin{itemize}
\item $\Psi_{\varepsilon}(x) \in L^{2}(\mathbb{R})$ whenever $\varepsilon \in (v_{\text{\text{min}}},v_-)$, and
\item $\Psi_{\varepsilon}(x) \in L^{\infty}(\mathbb{R})$ whenever $\varepsilon \in (v_-,\infty)$.
\end{itemize}
\end{problem}

Hence, the NU method for solving Problem \ref{schroProb} consists, roughly, of: $(i)$ transform SE into a GHE by means of a so-called universal change of variable, $(ii)$ transform GHE into a HDE by setting $\Psi_{\varepsilon}(x)=\chi(x;\varepsilon)y_{\varepsilon}(x)$ and, $(iii)$ find $y_{\varepsilon}(x)$ so that $\Psi_{\varepsilon}(x)$ is of $L^2(\mathbb{R})$ or $L^{\infty}(\mathbb{R})$ accordingly to particularities of regions \ref{R1} and \ref{R2}, respectively. In such manner, we will express the bound states in terms of the Classical Orthogonal Polynomials (OPs) by means of Theorem \ref{eigenTheo} and Table \ref{table1}; and, for the rest of the scattering states, in terms of general hypergeometric functions by means of Propositions \ref{gaussProp}, \ref{conlfuentProp}, \ref{conlfuentProp2} and \ref{hermiteProp} and Table \ref{linIndHyper}, mainly.

Two main issues arise by following this approach. Transforming SE into a GHE does not follow a constructive scheme. However, in many instances, it is achieved in a rather simple way by {\it straightforward} changes of variable. In addition, after this change of variable, say, $x=\xi(s)$, square-integrability and boundedness condition shift in shape and do not usually coincide for the reduced eigenvalue problem associated to the final HDE. Indeed, the eigenvalue problem that is addressed by Theorem \ref{eigenTheo} generally reads as follows:
\begin{problem}
\label{eigenProb}
Find all values of $\lambda \in \mathbb{C}$ for which
$$
\phi(s) \dfrac{\mathrm{d}^2y}{\mathrm{d}s^2}(s) + \psi(s) \dfrac{\mathrm{d}y}{\mathrm{d}s}(s) +\lambda y(s)=0, \quad s \in I,
$$
has non-trivial solutions, $y_{\lambda}(s)$, satisfying that $y_{\lambda}(s) \in L^{2}(I;\omega)$, where $\omega(s)$ is any solution of $$\dfrac{\mathrm{d}(\phi\omega)}{\mathrm{d}s}(s)=\psi(s)\omega(s),$$
i.e, its associated Pearson equation (in the ``old'' classical sense\footnote{The notion of {\it classicality} has been extended following the algebraic theory developed by Maroni, see, for instance, \cite[Chapter 8]{castillo2024first}. In the context of this paper, we work with the more traditional understanding of {\it classicality}, which is typically associated with the positive-definite cases: Jacobi, Laguerre, and Hermite polynomials. This traditional interpretation becomes particularly relevant when analyzing physical systems.}).
\end{problem} 
Now, given that the change of variable $x=\xi(s)$ is not generally linear, it will not be true that $$\Psi_{\varepsilon}(x)=\chi(x;\varepsilon)y_{\varepsilon}(x) \in L^2(\mathbb{R}) \iff y_{\varepsilon}(s) \in L^{2}(I;\omega).$$ However, we will be able to sort this out (together with further complexities) while dealing with each particular case. We will properly restate these issues when actually addressing the Eigenvalue Problem \ref{schroProb} in Section \ref{solving}.

In Section \ref{reduction}, the reduction of a GHE into an HDE is reviewed . In Section \ref{solving}, the HDE is solved, first, by means of Classical OPs and, secondly, by hypergeometric functions. These results are employed to address both the bound and the scattering states. In Section \ref{applications}, three potential are solved by means of the NU method and, finally, their peculiarities in terms of the method are discussed among further considerations.

\section{The reduction\label{reduction}}
As stated, we depart from the notion of generalized hypergeometric equation:
\begin{definition}
Let $I \subset \mathbb{R}$ be an open interval not necessarily bounded. A differential equation of the form
\begin{equation}
\dfrac{\mathrm{d^2}\Psi}{\mathrm{d}x^2}(x)+\frac{\widetilde{\psi}(x)}{\phi(x)}\dfrac{\mathrm{d}\Psi}{\mathrm{d}x}(x)+
\frac{\widetilde{\phi}(x)}{\phi^2(x)}\Psi(x)=0, \quad x \in  I,
\label{GHE}
\end{equation}
is said to be a generalized hypergeometric equation or GHE whenever polynomials $\phi(x)$, $\widetilde{\phi}(x) \in \mathcal{P}_2$ and $\widetilde{\psi}(x)\in \mathcal{P}_1$. 
\end{definition}

By setting $\Psi(x)=\chi(x)y(x)$ for an adequate $\chi(x)$, $y(x)$ solves another, hopefully simpler, GHE:
\begin{lemma}
Let $\chi(x)$ satisfy
\begin{equation}
    \dfrac{1}{\chi(x)}\dfrac{\mathrm{d}\chi}{\mathrm{d}x}(x) = \frac{\pi(x)}{\phi(x)}, \quad x \in I,
    \label{phi}
\end{equation}
for an arbitrary $\pi(x) \in \mathcal{P}_1$. Then, whenever $\Psi(x)$ is a solution of Eq. \eqref{GHE}, $y(x)$ given by $$\Psi(x)=\chi(x)y(x),$$ solves another GHE
\begin{equation} 
\dfrac{\mathrm{d}^2y}{\mathrm{d}x^2}(x)+\frac{\psi(x)}{\phi(x)}\dfrac{\mathrm{d}y}{\mathrm{d}x}(x)+
\frac{\overline{\phi}(x)}{\phi^2(x)}y(x)=0, \quad x \in I,
\label{GHEy}
\end{equation}
where
\begin{equation}
\begin{split}
&\psi(x)=\widetilde{\psi}(x)+2\pi(x) \in \mathcal{P}_1, \\
&\overline{\phi}(x)=\pi^2(x)+
\pi(x)\left(\widetilde{\psi}(x)-\dfrac{\mathrm{d}\phi}{\mathrm{d}x}(x)\right)+\dfrac{\mathrm{d}\pi}{\mathrm{d}x}(x)\phi(x)+\widetilde{\phi}(x) \in \mathcal{P}_2.
\end{split}
\label{tau,sigmabar}
\end{equation}
\end{lemma}
\begin{proof}
Assume $\Psi(x)$ solves Eq. \eqref{GHE} and put $\Psi(x)=\chi(x)y(x)$ into it. It yields
\begin{equation*}
    \begin{split}
    \dfrac{\mathrm{d}^2y}{\mathrm{d}x^2}(x)+&\left(\dfrac{2}{\chi(x)}\dfrac{\mathrm{d}\chi}{\mathrm{d}x}(x)+\frac{\widetilde{\psi}(x)}{\phi(x)}\right)\dfrac{\mathrm{d}y}{\mathrm{d}x}(x)+ \\
    &\left(\dfrac{1}{\chi(x)}\dfrac{\mathrm{d}^2\chi}{\mathrm{d}x^2}(x)+\dfrac{1}{\chi(x)}\dfrac{\mathrm{d}\chi}{\mathrm{d}x}(x)\frac{\widetilde{\psi}(x)}{\phi(x)}+
    \frac{\widetilde{\phi}(x)}{\phi^2(x)}\right)y(x)=0. 
    \end{split}
\end{equation*}
Now, consider Eq. \eqref{phi} and, from it, obtain
$$
\dfrac{1}{\chi(x)}\dfrac{\mathrm{d}^2\chi}{\mathrm{d}x^2}(x)=\dfrac{\mathrm{d}\pi}{\mathrm{d}x}(x)\dfrac{\phi(x)}{\phi^2(x)}+\dfrac{1}{\chi(x)}\dfrac{\mathrm{d}\chi}{\mathrm{d}x}(x)\dfrac{\pi(x)-\dfrac{\mathrm{d}\phi}{\mathrm{d}x}(x)}{\phi(x)}.
$$ 
Hence, above differential equation effectively transforms into Eq. \eqref{GHEy}. \end{proof}
Finally, we force new $\overline{\phi}(x)$ in Eq. \eqref{tau,sigmabar} to be proportional to $\phi(x)$ by means of $\pi(x)$ which, at this point, was an arbitrary linear polynomial. Hence, whenever, say, $\overline{\phi}(x)=\lambda \phi(x)$ for some $\lambda \in \mathbb{C}$, Eq. \eqref{GHEy} reduces to a hypergeometric differential equation:
\begin{definition}
Let $I \subset \mathbb{R}$ be an open interval not necessarily bounded. A differential equation of the form
\begin{equation} 
\phi(x) \dfrac{\mathrm{d}^2y}{\mathrm{d}x^2}(x) + \psi(x)\dfrac{\mathrm{d}y}{\mathrm{d}x}(x) +\lambda y(x)=0, \quad x \in I,
\label{HDE}
\end{equation}
is said to be a hypergeometric differential equation or HDE whenever polynomials $\phi(x) \in \mathcal{P}_2$, $\psi(x)\in \mathcal{P}_1$ and $\lambda \in \mathbb{C}$. Solutions of such differential equation are said to be functions of hypergeometric type.
\end{definition}
To choose the appropriate $\pi(x)$, consider the following Lemma:
\begin{lemma}
Define polynomial $P_2(x;k)$ depending parametrically on $k \in \mathbb{C}$ as
\begin{equation} 
P_2(x;k) \coloneqq \dfrac{1}{4}\left(\dfrac{\mathrm{d}\phi}{\mathrm{d}x}(x)-\widetilde{\psi}(x)\right)^2-\widetilde{\phi}(x)+k\phi(x) \in \mathcal{P}_2.
\label{P2}
\end{equation}
Then, whenever $k_0\in\mathbb{C}$ is such that the discriminant $\Delta P_2(k_0)$ vanishes, $\pi(x)$ given by
\begin{equation}
\pi(x) = \dfrac{1}{2}\left(\dfrac{\mathrm{d}\phi}{\mathrm{d}x}(x)-\widetilde{\psi}(x)\right) \pm \sqrt{P_2(z;k_0)},
\label{pichoice}
\end{equation}
is a polynomial in $\mathcal{P}_1$ satisfying that, whenever 
\begin{equation}
\lambda = k_0+\dfrac{\mathrm{d}\pi}{\mathrm{d}x}(x) \in \mathbb{C},
\label{lambda}
\end{equation}
$\overline{\phi}(x)=\lambda \phi(x)$.
\label{P2lemma}
\end{lemma}
\begin{proof}
It can be verified that $P_2(z;k) \in \mathcal{P}_2$ is a perfect square if and only if $\Delta P_2(k)$ is zero. Hence, for $\Delta P_2(k_0)=0$, $\pi(x) \in \mathcal{P}_1$. In addition, it solves
\begin{equation}
\label{piEq}
\overline{\phi}(x)=\lambda \phi(x) =\pi^2(x)+
\pi(x)\left(\widetilde{\psi}(x)-\dfrac{\mathrm{d}\phi}{\mathrm{d}x}(x)\right)+\dfrac{\mathrm{d}\pi}{\mathrm{d}x}(x)\phi(x)+\widetilde{\phi}(x).
\end{equation}
Thence, Eq. \eqref{GHE} is effectively reduced into Eq. \eqref{HDE}.
\end{proof}
Therefore, the standard NU approach for reducing a Eq. \eqref{GHE} into a Eq. \eqref{HDE} consists of: $(i)$ consider $P_2(x;k)$ given by Eq. \eqref{P2} and find $k_0 \in \mathbb{C}$ such that $\Delta P_2(k_0)=0$; $(ii)$ compute $\pi(x)$ using Eq. \eqref{pichoice}; and $(iii)$ compute $\chi(x)$, $\psi(x)$ and $\lambda$ using Eqs. \eqref{phi}, \eqref{tau,sigmabar} and \eqref{lambda}, respectively. Hence, solutions $\Psi(x)$ of Eq. \eqref{GHE} are found in the form $\Psi(x)=\chi(x)y(x)$ where $y(x)$ solves \eqref{HDE}. Of course, in certain occasions, step $(i)$ can be excluded and $\pi(x)$ be found by setting $\pi(x)=ax+b$, equating coefficients and solving the system.

Generally, condition $\Delta P_2(k_0)=0$ yields two possible $k_0 \in \mathbb{C}$ which, together with the ambiguity in the sign of $\pi(x)$ in Eq. \eqref{pichoice} gives a total of four different HDEs whose resolution is equivalent to that of the original GHE. In particular cases, there may be only two or, even, none. Subsequently, we will see that it will be vital to choose the adequate HDE to solve.

In following Section \ref{solving}, we will solve the HDE by means of, in particular, Classical OPs and, generally, hypergeometric functions. Those familiar with the standard theory (or uninterested by it) may skip most of the results (which are basic, minimal and included out of completeness) and only focus on Subsections \ref{boundstates} {\it ``Addressing the bound states''} and \ref{scattetingstates} {\it ``Addressing the scattering states''} where the Eigenvalue Problem \ref{schroProb} is actually addressed.

\section{Solving the eigenvalue problem \label{solving}}
In Subsection \ref{classicalOPs}, we distinguish Classical OPs as not only the simplest solutions but also the only square-integrable ones. Finally, in Subsection \ref{hyperFuncs}, we generally solve Eq. \eqref{HDE} in terms of arbitrary hypergeometric functions.
\subsection{Classical OPs\label{classicalOPs}}
The main characteristic of the HDE is its so-called {\it hypergeometric property}: given one of its solution, $y(x)$, any of its derivative also satisfies an HDE. Furthermore, in most cases, the solutions of the new HDE will be expressed in terms of those of the original one.
\begin{lemma}
\label{lemma:hypergeometricity}
Let $y(x)$ be a solution of Eq. \eqref{HDE}. Then, any of its m-th derivative, $$v_m(x) \coloneqq \dfrac{\mathrm{d}^{m}y}{\mathrm{d}x^m}(x), \quad m=0,1,2,\dots,$$ solve another HDE of the form
\begin{equation}
\phi(x) \dfrac{\mathrm{d}^2v}{\mathrm{d}x}(x) + \psi_m(x) \dfrac{\mathrm{d}v}{\mathrm{d}x}(x) +\mu_m v(x)=0, \quad x \in I,
	\label{HDEm}    
\end{equation}
where
\begin{equation}
    \psi_m(x) \coloneqq \psi(x)+m\dfrac{\mathrm{d}\phi}{\mathrm{d}x}(x) \in \mathcal P_1, \quad \mu_m \coloneqq \lambda+m\dfrac{\mathrm{d}\psi}{\mathrm{d}x}(x)+\frac{m(m-1)}{2}\dfrac{\mathrm{d}^2\phi}{\mathrm{d}x^2}(x)\in \mathbb{C}.
\label{taum,mum}
\end{equation}
Furthermore, let $w(x)$ be a solution of Eq. \eqref{HDEm}. Then, $$w(x)=\dfrac{\mathrm{d}^{m}y}{\mathrm{d}x^m}(x),$$ for some solution $y(x)$ of \eqref{HDE}, whenever $\mu_k \neq 0,\, k=0,1,\dots,m$.
\end{lemma}
\begin{proof}
Proceed by mathematical induction, the thesis is true for $m=0$. Now, assume that $v_{m-1}(x)$ satisfies
$$
\phi(x) \dfrac{\mathrm{d}^2v_{m-1}}{\mathrm{d}x^2}(x) + \psi_{m-1}(x)\dfrac{\mathrm{d}v_{m-1}}{\mathrm{d}x}(x)+\mu_{m-1} v_{m-1}(x)=0,
$$
where $\psi_{m-1}(x)$ and $\mu_{m-1}$ are given by Eq. \eqref{taum,mum}. Differentiating it and considering $$\dfrac{\mathrm{d}v_{m-1}}{\mathrm{d}x}(x)=v_{m}(x),$$ obtain
$$
\phi(x)\dfrac{\mathrm{d}^2v_{m}}{\mathrm{d}x^2}(x)\!+\left(\dfrac{\mathrm{d}\phi}{\mathrm{d}x}(x)+\psi_{m-1}(x)\!\right)\dfrac{\mathrm{d}v_{m}}{\mathrm{d}x}(x)+\left(\mu_{m-1}\!+\dfrac{\mathrm{d}\psi_{m-1}}{\mathrm{d}x}(x)\!\right)v_{m}(x)=0,
$$
and, so, $$\psi_{m}(x)=\dfrac{\mathrm{d}\phi}{\mathrm{d}x}(x)+\psi_{m-1}(x), \quad \mu_{m}=\mu_{m-1}+\dfrac{\mathrm{d}\psi_{m-1}}{\mathrm{d}x}(x),$$ which gives Eq. \eqref{taum,mum}. Hence, the first part.

Assume, now, that $w(x)$ is a solution of Eq. \eqref{HDEm} and $\mu_{m-1} \neq 0$. Then, by setting
\begin{equation*}
    v_{m-1}(x)=-\dfrac{1}{\mu_{m-1}}\left(\phi(x)\dfrac{\mathrm{d}w}{\mathrm{d}x}(x) + \psi_{m-1}(x)w(x)\right),
\end{equation*}
its derivative yields
\begin{equation*}
    \begin{split}
        \mu_{m-1}\dfrac{\mathrm{d}v_{m-1}}{\mathrm{d}x}(x)&=-\phi(x)\dfrac{\mathrm{d}^2w}{\mathrm{d}x^2}(x)-\left(\dfrac{\mathrm{d}\phi}{\mathrm{d}x}(x)-\psi_{m-1}(x)\!\right)\!\dfrac{\mathrm{d}w}{\mathrm{d}x}(x) - \dfrac{\mathrm{d}\psi_{m-1}}{\mathrm{d}x}(x)w(x) \\
        &=\mu_{m-1}w(x),
    \end{split}
\end{equation*}
thus, 
$$\dfrac{\mathrm{d}^2v_{m-1}}{\mathrm{d}x^2}(x)=\dfrac{\mathrm{d}w}{\mathrm{d}x}(x),$$ which gives
$$
\phi(x) \dfrac{\mathrm{d}^2v_{m-1}}{\mathrm{d}x^2}(x) + \psi_{m-1}(x)\dfrac{\mathrm{d}v_{m-1}}{\mathrm{d}x}(x)+\mu_{m-1} v_{m-1}(x) = 0,
$$
by its definition. Reiterating the argument for the rest of the $\mu_k \neq 0,\, k=0,1,\dots,m-2$, all the results follow.
\end{proof}
This gives us a way of constructing particular solutions of Eq. \eqref{HDE} for specific values of $\lambda \in \mathbb{C}$: whenever $\mu_n=0$, for $m=n$, Eq. \eqref{HDEm} has constant $w(x) \equiv c_n$ as a particular solution. By further assuming $\mu_k \neq 0,\, k=0,1,\dots,n-1$, $w(x)$ can be realised as $w(x)=v_{n}(x)$ for $y(x)$ some solution of original Eq. \eqref{HDE}. Thus, under such circumstances, Eq. \eqref{HDE} has a polynomial, $P_n(x)$, of degree $n$ as a particular solution. Such solutions are called polynomials of hypergeometric type. In addition, above conditions can be rewritten in a clearer way. Define
\begin{equation}
\lambda_n \coloneqq -n\left(\dfrac{\mathrm{d}\psi}{\mathrm{d}x}(x)+\frac{(n-1)}{2}\dfrac{\mathrm{d}^2\phi}{\mathrm{d}x^2}(x)\right), \quad n=0,1,\dots, 
\label{lambdan}
\end{equation}
then, $\mu_n=\lambda-\lambda_n$, $\mu_n=0$ is equivalent to $\lambda = \lambda_n$ and, under such hypothesis, non-vanishing $\mu_k$ is equivalent to $\lambda_n \neq \lambda_k$. Thus, we further require
\begin{equation}
0 \neq \mu_k = \mu_{nk} \coloneqq \lambda_n-\lambda_k = \left(k-n\right)\left(\dfrac{\mathrm{d}\psi}{\mathrm{d}x}(x)+\frac{k+n-1}{2}\dfrac{\mathrm{d}^2\phi}{\mathrm{d}x^2}(x)\right),
\label{munm}
\end{equation}
whenever $k=0,1,\dots,n-1$ for the polynomial solution to exist.

By expressing Eqs. \eqref{HDE} and \eqref{HDEm} in the self adjoint form, we find polynomials $P_n(x)$ explicitly by the so-called {\it Rodrigues formula}:
\begin{prop}
\label{rodriguesTheo}
Fix an arbitrary $n=0,1,\dots$ and assume $\lambda = \lambda_n$ as in Eq. \eqref{lambdan} and $\lambda_n \neq \lambda_k,\, k=0,1,\dots,n-1$. Then, there exists a polynomial solution, $P_n(x)$, of degree $n$ of Eq. \eqref{HDE}. Moreover, for $m = 0,1,\dots,n$,
\begin{equation*}
    \dfrac{\mathrm{d}^{m}P_n}{\mathrm{d}x^m}(x) = \frac{A_{mn}B_n}{\omega_m(x)}\dfrac{\mathrm{d}^{n-m}\omega_{n}}{\mathrm{d}x^{n-m}}(x),
\end{equation*}
where
\begin{equation*}
    A_{0n} \coloneqq 1, \quad A_{mn} \coloneqq (-1)^m \prod_{k=0}^{m-1}\mu_{nk}, \quad B_n \coloneqq \frac{1}{A_{nn}}c_n, \quad \omega_m(x) \coloneqq \phi^{m}(x)\omega(x),
\end{equation*}
$\mu_{nk}$ is given by Eq. \eqref{munm} and $\omega(x)$ is a solution of the so-called Pearson equation
\begin{equation}
    \dfrac{\mathrm{d}\omega_1}{\mathrm{d}x}(x)=\psi(x)\omega(x),
    \label{pearson}
\end{equation}
satisfying $\omega(x) \neq 0$, $x \in I$. In particular, the Rodrigues formula is given by
\begin{equation}
\label{rogriguesFormula}
    P_n(x)=\frac{B_n}{\omega(x)}\dfrac{\mathrm{d}^n\omega_n}{\mathrm{d}x^n}(x).
\end{equation}
\end{prop}
\begin{proof}
Existence is given by above discussion and Lemma \ref{lemma:hypergeometricity}. Now, set $$v_m(x)=\dfrac{\mathrm{d}^{m}y}{\mathrm{d}x^m}(x),$$ for $y(x)=P_n(x)$ satisfying $v_n(x)\equiv c_n$. Pearson Equation \eqref{pearson} yields $$\dfrac{\mathrm{d}(\phi\omega_m)}{\mathrm{d}x}(x)=\psi_m(x)\omega_m(x), \quad m = 0,1,\dots,n.$$ Hence, Eqs. \eqref{HDE} and \eqref{HDEm} can be written in a self-adjoint form
\begin{equation}
    \begin{split}
    \dfrac{\mathrm{d}}{\mathrm{d}x}\left(\phi(x)\omega(x)\dfrac{\mathrm{d}y}{\mathrm{d}x}(x)\right)+\lambda_n\omega(x)y(x)&=0, \\
    \dfrac{\mathrm{d}}{\mathrm{d}x}\left(\phi(x)\omega_m(x)\dfrac{\mathrm{d}v_m}{\mathrm{d}x}(x)\right)+\mu_{nm}\omega_m(x)v_m(x)&=0.
    \label{selfadjointm}
    \end{split}
\end{equation}
By considering $\phi(x)\omega_m(x)=\omega_{m+1}(x)$, Eq. \eqref{selfadjointm} reads
$$
\omega_m(x)v_m(x)=-\frac{1}{\mu_{nm}}\dfrac{\mathrm{d}(\omega_{m+1}v_{m+1})}{\mathrm{d}x}(x).
$$
Thus, reiterating
\begin{equation*}
    \begin{split}
        \omega_m(x)v_m(x)&=\frac{(-1)^1}{\mu_{nm}}\dfrac{\mathrm{d}(\omega_{m+1}v_{m+1})}{\mathrm{d}x}(x)=\frac{(-1)^2}{\mu_{nm} \mu_{n,m+1}}\dfrac{\mathrm{d}(\omega_{m+2}v_{m+2})}{\mathrm{d}x}(x) \\
        &=...=\frac{A_{mn}}{A_{nn}}\dfrac{\mathrm{d}^{n-m}(\omega_n v_n)}{\mathrm{d}^{n-m}x}(x)=A_{mn}B_n\dfrac{\mathrm{d}^{n-m}\omega_{n}}{\mathrm{d}x^{n-m}}(x).
    \end{split}
\end{equation*}
This proves all claims.
\end{proof}

Thus, considering Eq. \eqref{HDE} parametric on $\lambda$, we associate to it the sequence of polynomials $\left(P_n\right)^{\infty}_{n=0}$ given by the Rodrigues formula \eqref{rogriguesFormula} in terms of $\phi(x)$ and $\psi(x)$, through $\omega(x)$ given by Eq. \eqref{pearson}, by forcing $\lambda = \lambda_n$ and assuming $\lambda_k \neq \lambda_{k'}$ for $k \neq k'$ so that every $P_n(x)$ actually exists. Finally, by further imposing certain conditions on $\omega(x)$, sequence $\left(P_n\right)^{\infty}_{n=0}$ becomes a proper OPS with respect to the measure given by $\mathrm{d}\mu(x)=\omega(x)\mathrm{d}x$.
\begin{lemma}
Assume notation and hypotheses of Proposition \ref{rodriguesTheo} valid for each $n=0,1,\dots$ Let $\omega(x)$ satisfy
\begin{equation}
    \phi(x)\omega(x)x^k\Bigr|_{x=a} =    \phi(x)\omega(x)x^k\Bigr|_{x=b}=0,\quad k=0,1,\dots, \quad I=(a,b).
    \label{orthoPol}
\end{equation}
Then, 
\begin{equation*}
   \int_{I} P_k(x)P_{k'}(x)\omega(x)\mathrm{d}x=0, \quad k \neq k'.
\end{equation*}
\label{orthoTheo}
\end{lemma}
\begin{proof}
Say $P_k(x)$ and $P_{k'}(x)$ are associated to $\lambda_k \neq \lambda_{k'}$, respectively. Then, they solve
\begin{equation*}
    \begin{split}
    \dfrac{\mathrm{d}}{\mathrm{d}x}\!\left(\phi(x)\omega(x)\dfrac{\mathrm{d}P_{k}}{\mathrm{d}x}(x)\right)+\lambda_k\omega(x)P_k(x)&=0, \\
    \dfrac{\mathrm{d}}{\mathrm{d}x}\!\left(\phi(x)\omega(x)\dfrac{\mathrm{d}P_{k'}}{\mathrm{d}x}(x)\right)+\lambda_{k'}\omega(x)P_{k'}(x)&=0.
    \end{split}
\end{equation*}
Multiplying first equation by $P_{k'}(x)$ and second by $P_k(x)$, subtracting and integrating, it yields
\begin{equation*}
    \left(\lambda_{k'}-\lambda_k\right)\int_{I} P_k(x)P_{k'}(x)\omega(x)\mathrm{d}x=\phi(x)\omega(x)W\!\left[P_{k'},P_{k}\right]\!(x)\Bigr|_{\partial I},
\end{equation*}
where $$W[u,v](x )\coloneqq u(x)\dfrac{\mathrm{d}v}{\mathrm{d}x}(x)-v(x)\dfrac{\mathrm{d}u}{\mathrm{d}x}(x),$$ is the Wronskian. Since both $P_k(x)$ and $P_{k'}(x)$ are polynomials, $W\!\left[P_{k'},P_{k}\right]\!(x)$ is another polynomial. By Eq. \eqref{orthoPol}, right hand side is zero and $\lambda_k \neq \lambda_{k'}$ necessarily gives
\begin{equation*}
   \int_{I} P_k(x)P_{k'}(x)\omega(x)dx=0, \quad k \neq k',
\end{equation*}
as stated. This gives the result.
\end{proof}
In addition, we explicitly compute their squared norm
$$
d^2_n=\displaystyle \int_{I}P^2_n(x)\omega(x)\mathrm{d}x.
$$
\begin{lemma}
Assume notation and hypotheses of Lemma \ref{orthoTheo}. Let $d_{nk}^2$ denote
$$
d_{nk}^2=\int_I\omega_{k}(x)\left(\dfrac{\mathrm{d}^{k}P_n}{\mathrm{d}x^k}(x)\right)^2\mathrm{d}x.
$$
Then, $d^2_{k+1,n}=\mu_{kn}d^{2}_{kn}$ where $\mu_{kn}$ is given by Eq. \eqref{munm}. In particular,
$$c^2_n\int_I\omega_n(x)\mathrm{d}x = d_{nn}^2=d^2_n\prod^{n-1}_{k=0}\mu_{kn}.$$
Hence, $P_n(x) \in L^2(I;\omega)$ whenever $\omega_n(x) = \phi^n(x)\omega(x) \in L^{1}(I)$.
\label{intLemma}
\end{lemma}
\begin{proof}
Consider Eq. \eqref{selfadjointm} for $v_m(x)=P_m(x)$ and $m=k+1$, multiply it by $\dfrac{\mathrm{d}^{k}P_n}{\mathrm{d}x^k}(x)$ and integrate by parts over $I$. It yields
$$
\int^a_b\!\!\omega_{k+1}(x)\left(\dfrac{\mathrm{d}^{k+1}P_n}{\mathrm{d}x^{k+1}}(x)\right)^2\mathrm{d}x-
\mu_{kn}\int^a_b\!\!\omega_{k}(x)\left(\dfrac{\mathrm{d}^{k}P_n}{\mathrm{d}x^k}(x)\right)^2\mathrm{d}x=0,
$$
where the integrated term vanishes by Eq. \eqref{orthoPol}. This gives $d^2_{k+1,n}=\mu_{kn}d^2_{kn}$. The thesis is reached by reiterating the argument.
\end{proof}

\begin{table}[ht!]
\begin{center}
\begin{tabular}{c c c c}
 \toprule
 $P_n(x)$ & $P_n^{(\alpha,\beta)}(x)$ $\left(\alpha>-1, \; \beta>-1 \right)$ & $L_n^{\alpha}(x) \left( \alpha>-1 \right)$ & $H_n(x)$ \\ [0.5ex] 
 \midrule
 $I$ & $(-1,1)$ & $(0,\infty)$ & $(-\infty,\infty)$ \\ [0.2cm]
 \midrule
 $\omega(x)$ & $\left(1-x\right)^{\alpha}\left(1+x\right)^{\beta}$ & $x^{\alpha}e^{-x}$ & $e^{-x^2}$ \\ [0.2cm]
 \midrule
 $\phi(x)$ & $1-x^2$ & $x$ & $1$ \\ [0.2cm]
 \midrule
 $\psi(x)$ & $\beta-\alpha-\left(\alpha+\beta+2\right)x$ & $1+\alpha-x$ & $-2x$ \\ [0.2cm]
 \midrule
 $\lambda_n$ & $n\left(n+\alpha+\beta+1\right)$ & $n$ & $2n$ \\ [0.2cm]
 \midrule
 $\omega_n(x)$ & $\left(1-x\right)^{n+\alpha}\left(1+x\right)^{n+\beta}$ & $x^{n+\alpha}e^{-x}$ & $e^{-x^2}$ \\ [0.2cm]
 \midrule
 $c_n$ & $\frac{\Gamma(2n+\alpha+\beta+1)}{2^n\Gamma(n+\alpha+\beta+1)}$ & $1$ & $(-1)^n n!$ \\ [0.2cm]
 \midrule
 $d^2_n$ & $\frac{2^{\alpha+\beta+1} \Gamma(n+\alpha+1) \Gamma(n+\beta+1)}{n!(2n+\alpha+\beta+1)\Gamma(n+\alpha+\beta+1)}$ & $\frac{\Gamma(n+\alpha+1)}{n!}$ & $2^n n! \sqrt{\pi}$ \\ [0.2cm]
 \bottomrule
\end{tabular}
\caption{Data for the Classical OPs in the positive-definite case.}
\label{table1}
\end{center}
\end{table}

We will know Eq. \eqref{orthoPol} as condition on $\omega(x)$ for Classical OPs and, consequently, associated polynomials will be considered as Classical OPs. Auxiliary conditions $\omega(x)>0$ and $\phi(x)>0$ in $I$ are also imposed and, so, $d^2_n>0$. Under such definition, we may classify Classical OPs into the three canonical Jacobi, Laguerre and Hermite families which correspond to $P^{(\alpha,\beta)}_n(x)$, $L^{\alpha}_n(x)$ and $H_n(x)$, respectively, in Table \ref{table1} above.

Any Eq. \eqref{HDE} whose $\phi(x)$ has real but not double roots and whose $\psi(x)$ satisfies that $$\dfrac{\mathrm{d}\psi}{\mathrm{d}x}(x) \neq 0,$$ can be reduced, by means of linear changes of variable, into one of the three canonical HDE defined by $\phi(x)$ and $\psi(x)$ given in Table \ref{table1}. Further assuming conditions on the parameters and intervals indicated (first and second row), condition on each $\omega(x)$ for Classical OPs and condition $\phi^n(x)\omega(x) \in L^{1}(I)$ are satisfied. Namely, $a=-1,\,b=1$ and $\alpha,\beta>-1$ for Jacobi polynomials, $a=0,\,b=\infty$ and $\alpha>-1$ for Laguerre polynomials and $a=-\infty,\,b=\infty$ for Hermite polynomials. In addition, $\lambda_k \neq \lambda_{k'}$ for $k \neq k'$ for the three families. Hence, by Proposition \ref{rodriguesTheo}, they are well defined for each degree and given explicitly by Eq. \eqref{rogriguesFormula}, by Lemma \ref{orthoTheo}, are orthogonal and, by Lemma \ref{intLemma}, squared integrable with respect to $\mathrm{d}\mu$. Finally, $d^2_n$ (last row of Table 1) is included.

We have left behind case $\phi(x)=\left(x-c\right)^2$ on purpose. In this scenario, the associated polynomials are the Bessel polynomials. However, its $\omega(x)$ does not satisfy Eq. \eqref{orthoPol} and orthogonality cannot be understood by means of a positive-definite measure. Actually, in a classical context, it requires integrating over the unit circle in the complex plane. Despite not being completely exhaustive, in a mathematical sense, we will see that, for the purpose of solving Eigenvalue Problem \ref{schroProb}, the analysis carried out is sufficient: vast majority of physical system reduce to Jacobi, Laguerre or Hermite HDE and subsequent results that will be employed for addressing the bound states, namely, Theorem \ref{eigenTheo}, rely on positive definiteness of the measure. 

That said, it is worth noting that though sufficient it is, perhaps, outdated. For a modern approach to the study and characterization of classical families based on the algebraic theory of orthogonal polynomials developed by Maroni, reader is referred to \cite{castillo2021}. In that context, orthogonality is understood with respect to a functional belonging to an adequate dual space and, so, the four families of polynomials namely, Jacobi, Laguerre, Hermite and Bessel, are treated on an equal footing and considered as classical\footnote{As Maroni said
 in an interview to K. Castillo and Z. da Rocha when asked about Bessel polynomials: ``comme dans le roman
 d’Alexandre Dumas, les trois mousquetaires étaient quatre en réalité''.}. However, in our context, classical will be understood in the traditional sense, which is closely related to of positive definiteness. For simplicity, we shall refer to Classical OPs as what is known as the ``old'' Classical OPs in current literature, see \cite{castillo2024first}.

\subsubsection{Addressing the bound states\label{boundstates}}
Firstly, recall shape of Problem \ref{eigenProb}. This problem is solved by means of the following:
\begin{theorem}
Assume notation employed in Problem \ref{eigenProb}. Assume further hypotheses of Lemma \ref{intLemma} and an associated HDE linearly reducible to the canonical forms in Table \ref{table1}. Then, non-trivial solutions, $y_{\lambda}(s)$, satisfying that $y_{\lambda}(s) \in L^{2}(I;\omega)$ exist only when
\begin{equation}
\label{eigenEq}
    \lambda = \lambda_n = -n\left(\dfrac{\mathrm{d}\psi}{\mathrm{d}s}(s)+\frac{(n-1)}{2}\dfrac{\mathrm{d}^2\phi}{\mathrm{d}s^2}(s)\right), \quad n=0,1,\dots,
\end{equation}
and, denoting $y_{\lambda}(s)=y_{\lambda_n}(s) \coloneqq P_{n}(s)$, are given by Rodrigues formula, i.e, Eq. \eqref{rogriguesFormula},
$$
P_n(s)=\frac{B_n}{\omega(s)}\dfrac{\mathrm{d}^n(\phi^n\omega)}{\mathrm{d}s^n}(s).
$$
Hence, they are the Classical OPs associated to $\omega(s)$.
\label{eigenTheo}
\end{theorem}
\begin{proof}
By above discussion and Lemma \ref{intLemma}, pairs $(\lambda, P_n(s))$ are solutions of Problem \ref{eigenProb}.

Assume, now, that there exists some other $(\lambda,y_{\lambda}(s))$ solving Problem \ref{eigenProb}. Then, for an arbitrary $n = 0,1,\dots$,
\begin{equation*}
    \begin{split}
    \dfrac{\mathrm{d}}{\mathrm{d}s}\!\left(\phi(s)\omega(s)\dfrac{\mathrm{d}y_{\lambda}}{\mathrm{d}s}(s)\right)+\lambda_k\omega(s)y_{\lambda}(s)&=0, \\
    \dfrac{\mathrm{d}}{\mathrm{d}s}\!\left(\phi(s)\omega(s)\dfrac{\mathrm{d}P_n}{\mathrm{d}s}(s)\right)+\lambda_n\omega(s)P_n(s)&=0,
    \end{split}
\end{equation*}
both hold. Multiply first equality by $P_n(s)$ and second by $y_{\lambda}(s)$, subtract and integrate over $(x_1,x_2) \subsetneq I$ so that there are no singular points. Obtain,
\begin{equation}
    \label{theo31step1}
    (\lambda-\lambda_n)\displaystyle\int^{x_2}_{x_1}y_{\lambda}(u)P_n(u)\omega(u)\mathrm{d}u+\phi(s)\omega(s)W\!\left[P_{n},y_{\lambda}\right]\!(s)\Bigr|^{x_2}_{x_1}=0.
\end{equation}
By assumption, both $y_{\lambda}(s)$ and $P_n(s)$ are of $L^2(I;\omega)$ and, so, we may take the limit in Eq. \eqref{theo31step1} as $x_1 \to a$ and $x_2 \to b$. Hence, write
\begin{equation}
    \begin{split}
    \lim_{s \to a}\phi(s)\omega(s)W\!\left[P_{n},y_{\lambda}\right]\!(s)=c_1, \\
    \lim_{s \to b}\phi(s)\omega(x)W\!\left[P_{n},y_{\lambda}\right]\!(s)=c_2,
    \end{split}
\label{asympTheo31}
\end{equation}
for some finite constants $c_1$ and $c_2$. Now, assume both constants are vanishing. Then, whenever $\lambda \neq \lambda_n$, Eq. \eqref{theo31step1} yields
$$\displaystyle\int^{b}_{a}y_{\lambda}(s)P_n(s)\omega(s)\mathrm{d}s = 0, \quad n=0,1,\dots,$$ which implies $y_{\lambda}(s) \equiv 0$ as $\left(P_n(s)\right)_{n}$ is a complete set in $L^2(I;\omega)$; a contradiction. Whenever $\lambda = \lambda_n$, Eq. \eqref{theo31step1} gives $$\phi(s)\omega(s)W\!\left[P_{n},y_{\lambda}\right]\!(s) \equiv c,$$ as $x_1$ and $x_2$ were arbitrary. Hence, $c=c_1=c_2$ and, if they are vanishing, it necessarily follows $$W\!\left[P_{n},y_{\lambda}\right]\!(s)=0,$$
and, so, they are linearly dependent; a contradiction, again. Therefore, it suffices to prove that $c_1$ and $c_2$ are vanishing.

Indeed, by the relation $$\dfrac{\mathrm{d}}{\mathrm{d}s}\left(\dfrac{y_{\lambda}}{P_n}\right)(s)=\dfrac{W\!\left[P_{n},y_{\lambda}\right]\!(s)}{P^2_n(s)},$$ write $$y_{\lambda}(s)=P_n(s)\left(\dfrac{y_{\lambda}(s_0)}{P_n(s_0)}+\int^{s}_{s_0}\dfrac{W\!\left[P_{n},y_{\lambda}\right]\!(u)}{P^2_n(u)}\mathrm{d}u\right),$$ for some $s_0 \in I$ lying to the right of the zeros of $P_n(s)$. Introducing asymptotics given by Eq. \eqref{asympTheo31} in above $y_{\lambda}(s)$ and integrating for the particular data given by, for instance, Table \ref{table1}, it necessarily follows that $c_2$ must vanish so that $y_{\lambda} \in L^2(I;\omega)$. Similarly, one can obtain that $c_1$ must vanish to retain square-integrability. For the specific computations, reader is refered to \cite[\S9.2, pp. 69-71]{Nikiforov1988SpecialFO}. Hence, the result. \end{proof}
Thence, Classical OPs are distinguished among the rest of solutions not only by their simplicity but also because they conform the only square-integrable solutions, under appropriate conditions. 

Now, let us revisit the method itself. Depart from Eq. \eqref{schroadi} and attempt to transform it into a GHE as in Eq. \eqref{GHE}. Consider a {\it sufficiently nice} transformation $x=\xi(s)$ for a map $\xi: I \to \mathbb{R}$ and denote 
\begin{equation*}
\rho(s) \coloneqq \dfrac{\mathrm{d}\xi}{\mathrm{d}s}(s).     
\end{equation*}
Hence, from Eq. \eqref{schroadi}, obtain
\begin{equation}
   \dfrac{\mathrm{d}^2\Psi}{\mathrm{d}s^2}(s)-\dfrac{1}{\rho(s)}\dfrac{\mathrm{d}\rho}{\mathrm{d}s}(s)\dfrac{\mathrm{d}\Psi}{\mathrm{d}s}(s)+\rho^2(s)\left(\varepsilon-v(s)\right)\Psi(s)=0 ,\quad s \in I.
   \label{changedschro}
\end{equation}
Then, Eq. \eqref{changedschro} corresponds to an Eq. \eqref{GHE} whenever
\begin{equation}
    \dfrac{1}{\rho(s)}\dfrac{\mathrm{d}\rho}{\mathrm{d}s}(s)=-\dfrac{\widetilde{\psi}(s)}{\phi(s)}, \quad \rho^2(s)\left(\varepsilon-v(s)\right)=\dfrac{\widetilde{\phi}(s;\varepsilon)}{\phi^2(s)},
    \label{conds}
\end{equation}
for some polynomials $\phi(s)$, $\widetilde{\phi}(s;\varepsilon) \in \mathcal{P}_2$ and $\widetilde{\psi}(s)\in \mathcal{P}_1$. Generally, $\xi(s)$ cannot be found in this way but it does hint us towards the adequate choice. For instance, potential $v(x)$ and $$\dfrac{1}{\rho(s)}\dfrac{\mathrm{d}\rho}{\mathrm{d}s}(s)$$ should be rational in $\tau(x) \coloneqq \xi^{-1}(x)$ and $s$, respectively. Now, assume Eq. \eqref{conds} is satisfied and write
\begin{equation}
\dfrac{\mathrm{d}^2\Psi}{\mathrm{d}s^2}(s)+\frac{\widetilde{\psi}(s)}{\phi(s)}\dfrac{\mathrm{d}\Psi}{\mathrm{d}s}(s)+
\frac{\widetilde{\phi}(s;\varepsilon)}{\phi^2(s)}\Psi(s)=0, \quad s \in  I.
\label{GHEschro}
\end{equation}
Set $\Psi(s)=\chi(s;\varepsilon)y(s)$, following NU approach for reducing GHE into an Eq. \eqref{HDE}, and, if successful, obtain 
\begin{equation}
    \phi(s) \dfrac{\mathrm{d}^2y}{\mathrm{d}s^2}(s) + \psi(s;\varepsilon) \dfrac{\mathrm{d}y}{\mathrm{d}s}(s)+\lambda(\varepsilon) y(s)=0, \quad s \in I.
\label{HDEschro}
\end{equation}
Apply, if possible, Theorem \ref{eigenTheo} to solve the associated eigenvalue problem in terms of $\varepsilon$ (parametric dependence of $\psi$ with respect to $\varepsilon$ can be sorted out) and find all pairs $(P_n(s),\varepsilon_n)_{n}$ solving Eq. \eqref{HDEschro} such that $P_n(s)\sqrt{\omega(s;\varepsilon)} \in L^2(I)$. Now, define $\widetilde{\omega}(s)$ as any function satisfying
\begin{equation}
    \dfrac{\mathrm{d}(\phi\widetilde{\omega})}{\mathrm{d}s}(s)=\widetilde{\psi}(s)\widetilde{\omega}(s),
    \label{rhotilde}
\end{equation}
associated to the self adjoint form of Eq. \eqref{GHE}. Then, by Eqs. \eqref{phi} and \eqref{pearson}, it holds $\omega(s)=\widetilde{\omega}(s)\chi^2(s)$. Finally, denoting $\Psi_n(x) \coloneqq \chi(\tau(x);\varepsilon_n)P_n(\tau(x))$, write
\begin{equation}
\begin{split}
+\infty > \int_I P^2_n(s)\omega(s;\varepsilon)\mathrm{d}s &= \int_I P^2_n(s)\chi^2(s;\varepsilon_n)\widetilde{\omega}(s;\varepsilon_n)\mathrm{d}s\\
&=\int_{\mathbb{R}}\Psi^2_n(x)\widetilde{\omega}(x;\varepsilon_n)\left|\tau'(x)\right|\mathrm{d}x,
\end{split}
\label{nonintegrability}
\end{equation}
which is not equivalent, in general, to $\Psi_n(x) \in L^2(\mathbb{R})$, the condition we aim for from the beginning. However, in certain occasions, we can use the following result:
\begin{coro}
\label{coroInt}
Assume notation and discussion of this epigraph. Assume further that NU reduction is successful and Theorem \ref{eigenTheo} is applicable. Then, if 
\begin{equation}
    \left|\widetilde{\omega}(x;\varepsilon_n)\tau'(x)\right| < C, \quad C>0, \quad x \in \mathbb{R},
\label{estimate}
\end{equation}
being $\tau(x) \coloneqq \xi^{-1}(x)$, the only possible bound states associated to Eigenvalue Problem \ref{schroProb} are the pairs $(\Psi_n(x),\varepsilon_n)$ where $$\Psi_n(x)=\chi(\tau(x);\varepsilon_n)P_n(\tau(x)),$$
function $\chi(s;\varepsilon)$ is given by the NU reduction associated to Eq. \eqref{GHEschro}, $P_n(s)$ are the Classical OPs associated to $\omega(s;\varepsilon)$ from Pearson equation of Eq. \eqref{HDEschro} and, finally, the bound eigenenergies, $\varepsilon_n \in (v_{\text{min}},v_-)$, are given implicitly by
\begin{equation*}
\label{eigenvareps}
    \lambda(\varepsilon_n)=\lambda_n(\varepsilon_n)=-n\left(\dfrac{\partial \psi}{\partial s}(s;\varepsilon_n)+\frac{(n-1)}{2}\dfrac{\mathrm{d}^2\phi}{\mathrm{d}s^2}(s)\right), \quad n=0,1,\dots
\end{equation*}
\end{coro}
\begin{proof}
Assume $\left(\Psi_{\varepsilon}(x),\varepsilon\right)$ is another bound state that cannot be written as stated. Then, write $\Psi_{\varepsilon}(x)=\chi(x;\varepsilon)y_{\varepsilon}(x)$ and, so, $(y_{\varepsilon}(x),\varepsilon)$ solves Eq. \eqref{HDEschro}. In addition, Eq. \eqref{nonintegrability} for $P_n(s)=y_{\varepsilon}(s)$ together with estimate \eqref{estimate}, gives $y_{\varepsilon}(s) \in L^2(I;\omega(\cdot ; \varepsilon))$ as $\left(\Psi_{\varepsilon}(x),\varepsilon\right)$ is a bound state. By Theorem \ref{eigenTheo}, $y_{\varepsilon}(s)=P_n(s)$ for some $n=0,1,\dots$; a contradiction. Hence, the result.
\end{proof}
Notice that we do not state that they are, indeed, bound states. However, in most of the cases, it will be straightforward to check $\Psi_n(x) \in L^2(\mathbb{R})$. As to the success of the reduction, there is not much we can say about it in general terms. Now, applicability of Theorem \ref{eigenTheo} can be obtained by choosing the adequate reduced HDE:
\begin{lemma}
\label{easesChoiceLemma}
Let $\omega(s)$ satisfy Eq. \eqref{orthoPol} for Classical OPs on $I$. Then, $\psi(s)$ vanishes at some point of $I$ and $$\dfrac{\mathrm{d}\psi}{\mathrm{d}s}(s)<0.$$
\end{lemma}
\begin{proof}
From Rodrigues formula \eqref{rogriguesFormula} together with Pearson Eq. \eqref{pearson}, $P_1(s) \propto \psi(s)$. Despite not proven in this paper, zeros of Classical OPs $(P_n)_n$ for the positive-definite case lie on $I$. This proves first assertion.

From Lemma \ref{intLemma}, $$d^2_{11}=\lambda_1d^2_1=-\dfrac{\mathrm{d}\psi}{\mathrm{d}s}(s)d^2_1.$$ By auxiliary conditions $\omega(s)>0$ and $\phi(s)>0$, $d^2_{11}>0$ and $d^2_1>0$ and, so, $$\dfrac{\mathrm{d}\psi}{\mathrm{d}s}(s)<0.$$ Hence, the result follows.
\end{proof}
To finalize this section, we present one further consideration that is useful when dealing with the particular physical systems discussed in Section \ref{applications}.

\begin{lemma}
Assume notation of Section \ref{reduction}. Assume further that $$\widetilde{\psi}(s)=\dfrac{\mathrm{d}\phi}{\mathrm{d}s}(s).$$ Then, $\omega(s)=\chi^2(s)$. Hence $\widetilde{\omega}(s) \equiv 1$.
\label{lemmaEasesInt}
\end{lemma}
\begin{proof}
By Eqs. \eqref{phi} and \eqref{pearson}, the result follows.
\end{proof}
Previous Corollary \ref{coroInt} is more applicable under such circumstances. This is the case of the Quantum Harmonic Oscillator or the Rosen-Morse II potential, for instance, see Section \ref{applications}. Despite these nuisances, the NU method achieves to solve the discrete spectrum of many systems in terms of Classical OPs.

\subsection{Hypergeometric functions\label{hyperFuncs}}
Next, we generalise Rodrigues formula \eqref{rogriguesFormula} for arbitrary values of $\lambda \in \mathbb{C}$ to construct, generally, the two linearly independent solutions of Eq. \eqref{HDE}. The idea is as follows: assume $\lambda=\lambda_n$ for some $n=0,1,\dots$ given by Eq. \eqref{lambdan}. Then, by Proposition \ref{rodriguesTheo}, polynomial $P_n(x)$ given by $$P_n(x)=\frac{B_n}{\omega(x)}\dfrac{\mathrm{d}^n(\phi^n\omega)}{\mathrm{d}x^n}(x),$$ solves Eq. \eqref{HDE}. By Cauchy's integral formula, it can be written as
\begin{equation}
P_n(x) = \frac{C_n}{\omega(x)}\int_{C}\frac{\phi^{n}(u)\omega(u)}{\left(u-x\right)^{n+1}}\mathrm{d}u,
\label{cauchyIntegral}
\end{equation}
where $C$ is an arbitrary closed contour surrounding $x \in \mathbb{R}$ and $C_n=\dfrac{B_n n!}{2\pi i}$. Now, assume that, for an arbitrary $\nu \in \mathbb{C}$,
\begin{equation}
    \lambda = \lambda_{\nu} \coloneqq -\nu \dfrac{\mathrm{d}\psi}{\mathrm{d}x}(x)-\frac{\nu \left(\nu-1\right)}{2} \dfrac{\mathrm{d}^2\phi}{\mathrm{d}x^2}(x).
    \label{lambdanu}
\end{equation}
Then, under {\it adequate conditions}, substituting $n \in \mathbb{N}_0$ in Eq. \eqref{cauchyIntegral} by previous $\nu \in \mathbb{C}$ yields a solution for Eq. \eqref{HDE}.
\begin{lemma}
\label{particularSolutionTheo}
Assume that $\lambda = \lambda_{\nu}$ for some $\nu \in \mathbb{C}$ as in Eq. \eqref{lambdanu} and $\omega(x)$ is a solution of Pearson Eq. \eqref{pearson}. Define
\begin{equation}
    v(x) = \int_{C}\frac{\phi^{\nu}(u)\omega(u)}{\left(u-x\right)^{\nu+1}}\mathrm{d}u.
    \label{uPart}
\end{equation}
Then, for a normalizing constant $C_{\nu}$,
\begin{equation}
    y(x) = y_{\nu}(x) \coloneqq \frac{C_{\nu}}{\omega(x)}v(x),
\end{equation}
solves Eq. \eqref{HDE} provided that
\begin{enumerate}[label=(\Roman*)]
\item \label{firstCond} we may interchange differentiation with respect to $x$ and integration with respect to $u$ in calculating the derivatives, i.e,
\begin{gather*}
\dfrac{\mathrm{d}v}{\mathrm{d}x}(x)=\left(\nu+1\right)\int_{C}\frac{\phi^{\nu}(u)\omega(u)}{\left(u-x\right)^{\nu+2}}\mathrm{d}u, \\
\dfrac{\mathrm{d}^2v(x)}{\mathrm{d}x^2}   =\left(\nu+1\right)\left(\nu+2\right)\int_{C}\frac{\phi^{\nu}(u)\omega(u)}{\left(u-x\right)^{\nu+3}}\mathrm{d}u,
\end{gather*}
and,
\item \label{secondCond} contour $C$ is chosen so that
\begin{equation*}
    \frac{\phi^{\nu+1}(u)\omega(u)}{\left(u-x\right)^{\nu+2}}\Biggr|^{u_1}_{u_2} = 0,
\end{equation*}
where $u_1$ and $u_2$ are the endpoints of $C$.
\end{enumerate}
\label{lemmaParticularSolutions}
\end{lemma}
\begin{proof}
Define, analogously to polynomial case,
\begin{equation*}
    \omega_{\nu}(x)=\phi^{\nu}(x)\omega(x), \quad \psi_{\nu}(x)=\psi(x)+\nu\phi'(x),
\end{equation*}
and, so, 
\begin{equation*}
    \dfrac{\mathrm{d}(\phi\omega_{\nu})}{\mathrm{d}x}(x)=\psi_{\nu}(x)\omega_{\nu}(x).
\end{equation*}
Multiply above equation by $(u-x)^{-\nu-2}$ and integrate by parts over contour $C$ satisfying Condition \ref{secondCond}. It yields
\begin{equation*}
\frac{\phi^{\nu+1}(u)\omega(u)}{\left(u-x\right)^{\nu+2}}\Biggr|^{u_1}_{u_2} + \left(\nu+2\right)\int_{C}\frac{\phi(u)\omega_{\nu}(u)}{\left(u-x\right)^{\nu+3}}\mathrm{d}u=\int_{C}\frac{\psi_{\nu}(u)\omega_{\nu}(u)}{\left(u-x\right)^{\nu+2}}\mathrm{d}u.
\end{equation*}
Integrated term vanishes and, by expanding $\phi(s)$ and $\psi_{\nu}(s)$ in powers of $u-x$, obtain
\begin{equation*}
\begin{split}
\dfrac{\phi(x)}{\nu+1}\dfrac{\mathrm{d}^2v}{\mathrm{d}x^2}(x)&+\dfrac{\nu+2}{\nu+1}\dfrac{\mathrm{d}\phi}{\mathrm{d}x}(x)\dfrac{\mathrm{d}v}{\mathrm{d}x}(x)+\dfrac{\nu+2}{2}\dfrac{\mathrm{d}^2\phi}{\mathrm{d}x^2}(x)v(x) \\
&=\dfrac{\psi_{\nu}(x)}{\nu+1}\dfrac{\mathrm{d}v}{\mathrm{d}x}(x)+\psi'_{\nu}(x)v(x),
\end{split}
\end{equation*}
where Condition \ref{firstCond} is used.
Inserting explicit form of $\psi_{\nu}(x)$, rewrite
\begin{equation}
\begin{split}
    \phi(x)\dfrac{\mathrm{d}^2v}{\mathrm{d}x^2}(x)&+\left(2\dfrac{\mathrm{d}\phi}{\mathrm{d}x}(x)-\psi(x)\right)\dfrac{\mathrm{d}v}{\mathrm{d}x}(x)- \\
    &\left(\nu+1\right)\left(\dfrac{\mathrm{d}\psi}{\mathrm{d}x}(x)+\dfrac{\nu-2}{2}\dfrac{\mathrm{d}^2\phi}{\mathrm{d}x^2}(x)\right)v(x)=0.
\end{split}
\label{midstep}
\end{equation}
Then, Eq. \eqref{midstep} together with $\omega(x)y(x) = C_{\nu} u(x)$ leads, after straightforward computations, to $$\dfrac{\mathrm{d}}{\mathrm{d}x}\!\left(\phi\omega\dfrac{\mathrm{d}y}{\mathrm{d}x}\right)(x)=-\lambda_{\nu}\omega(x)y(x).$$ Hence, the result.
\end{proof}
Applicability of Lemma \ref{lemmaParticularSolutions} is subjected to Conditions \ref{firstCond} and \ref{secondCond}. First one is sorted out by means of Theorem \ref{theoA.1}. As for the second, simple contours can be found under convenient restrictions on the coefficients of Eq. \eqref{HDE}. Subsequently, these restrictions are removed by using analytic continuation: derivative of functions $y(x)$ of hypergeometric type are also of hypergeometric type. It follows that by continuing $y(x)$ analytically we obtain analytic continuations of its derivative with respect to $x$ and the parameters. By the principle of analytic continuation, solutions $y(x)$ of Eq. \eqref{HDE} constructed following Lemma \ref{lemmaParticularSolutions} with particular restrictions on the parameters will continue to satisfy Eq. \eqref{HDE} in the whole region in which its left-hand side is analytic. Hence, we will be able to construct particular solutions for arbitrary values of $\lambda \in \mathbb{C}$.

In addition, we will transform further original Eq. \eqref{HDE} into another HDE by means of the NU reduction discussed in Section \ref{reduction}: any HDE can be seen as a GHE with the particularity that $\widetilde{\phi}(z) = \lambda \phi(z)$ and $\widetilde{\psi}(z)=\psi(z)$. Thus, following Lemma \ref{lemmaParticularSolutions}, we will further construct particular solutions for the new HDE, subsequently, map them to, in principle, a new linearly independent solution of the original HDE and, thus, completely characterising the space solution of Eq. \eqref{HDE}.

Again, as for the case for the Classical OPs, we will contempt ourselves by examining the cases related to an HDE whose $\phi(x)$ has real but not double roots and whose $\psi(x)$ satisfies $$\dfrac{\mathrm{d}\psi}{\mathrm{d}x}(x) \neq 0.$$ According to the degree of $\phi(x)$, we distinguish three different canonical forms, namely, the Gauss hypergeometric Eq. \eqref{hyperEq}, the confluent hypergeometric Eq. \eqref{confluentHyperEq} and the Hermite Eq. \eqref{hermiteEq} which originate as $\phi(x)$ has degree 2, 1 and 0, respectively. In Subsections \ref{gauss}, \ref{confluent} and \ref{hermite}, we deal with the problem of finding their respective two linearly independent solutions by following above construction. Finally, in Subsection \ref{scattetingstates}, we discuss their properties in relation to solving the scattering states region associated to Eigenvalue Problem \ref{schroProb}.

\subsubsection{Gauss hypergeometric equation\label{gauss}}
Whenever $\phi(x)\in \mathcal{P}_2,$ Eq. \eqref{HDE} can be expressed as a Gauss hypergeometric equation:
\begin{definition}
Let $I \subset \mathbb{R}$ be an open interval not necessarily bounded. A Gauss hypergeometric equation or GHDE is a differential equation of the form
\begin{equation} 
x\left(1-x\right)\dfrac{\mathrm{d}^2y}{\mathrm{d}x^2}(x)+\left(c-\left(a+b+1\right)x\right)\dfrac{\mathrm{d}y}{\mathrm{d}x}(x)-a b y(x)=0, \quad x \in I,
\label{hyperEq}
\end{equation}
where $a, b, c \in \mathbb{C}$ are arbitrary. In addition, $\omega(x)$ and $\nu \in \mathbb{C}$ for its associated Eqs. \eqref{pearson} and \eqref{lambdanu} are given by
\begin{equation}
        \omega(x)=x^{c-1}\left(1-x\right)^{a+b-c}, \quad \nu = -a, \, -b,
        \label{gaussPar}
\end{equation}
respectively.
\end{definition}
We find two linearly independent solutions of Eq. \eqref{hyperEq}. Firstly, we obtain some particular solutions by means of Lemma \ref{particularSolutionTheo} under adequate restrictions on the parameters $a, b, c \in \mathbb{C}$ and the open interval $I \subset \mathbb{R}$.
\begin{lemma}
\label{partSolGauss}
Consider Eq. \eqref{hyperEq} for $I=(0,1)$. Then, the following three contours $u_i:[0,1] \to \mathbb{C}$ for $i=1,2,3$ satisfy Condition \ref{firstCond} associated to such GHDE under the parametric restrictions specified:
\begin{align*}
    & u_1(t)=xt, & \Re(c)>\Re(a)>2, \\
    & u_2(t)=1-(1-x)t, & \Re(c)<\Re(b)+1, \, \Re(a)>2, \\
    & u_3(t)=x/t, & \Re(b)>1,\,\Re(a)>2.
\end{align*}
Their associated particular solutions, $y_i(z)$, are written

\begin{align*}
&y_1(x) = \pFq[skip=4]{2}{1}{a,b}{c}{x}, \label{1part} \nonumber\\
&y_2(x) = \pFq[skip=4]{2}{1}{a,b}{a+b-c+1}{1-x}, \\
&y_3(x) = x^{-a}\pFq[skip=4]{2}{1}{a,a-c+1}{a-b+1}{\dfrac{1}{x}},
\end{align*}
respectively, where\footnote{We also write ${}_2F_1(a,b,c;x) \coloneqq \pFq[skip=4]{2}{1}{a,b}{c}{x},$ when needed.}
\begin{equation}
    \pFq[skip=4]{2}{1}{a,b}{c}{x}=\frac{\Gamma(c)}{\Gamma(a)\Gamma(c-a)} \left(1-x\right)^{c-a-b} \int^{1}_{0}\dfrac{t^{c-a-1}\left(1-t\right)^{a-1}}{\left(1-xt\right)^{b}}\mathrm{d}t,
    \label{hypergeometricFunction}
\end{equation}
is the so-called Gauss hypergeometric function or, simply, hypergeometric function. Here, $c(z)$ stands for the gamma function.
\end{lemma}
\begin{proof}
For Eq. \eqref{hyperEq}, Condition \ref{firstCond} translates to
$$
p(u_2)-p(u_1) \coloneqq u^{c-a}\left(1-u\right)^{b-c+1}\left(u-x\right)^{a-2}\Bigr|^{u_2}_{u_1}=0,
$$
where $\omega(x)$ and $\nu \in \mathbb{C}$ are given by Eq. \eqref{gaussPar}. Under such restrictions, both $p(u_1)$ and $p(u_2)$ vanish for each one of the contours. Now, define
$$
f(t,x) = t^{c-a-1}\left(1-t\right)^{a-1}\left(1-xt\right)^{-b}, \quad (t,x) \in (0,1) \times (0,1) \coloneqq \Omega.
$$
Hence, $f(t,x) \in C(\Omega)$ and $f(t_0,x) \in \mathcal{H}(0,1)$ for every $t_0 \in (0,1)$: Theorem \ref{theoA.1} can be employed for Condition \ref{secondCond}.

Accordingly to Lemma \ref{particularSolutionTheo}, ${}_2F_1(a,b,c;x)$ given by Eq. \eqref{hypergeometricFunction} constitutes a particular solution of Eq. \eqref{hyperEq} for $a=0$ and $b=1$ whenever $\Re(c)>\Re(a)>2$, for instance. The rest of the particular solutions are obtained similarly from the rest of the contours. 

Finally, constant is chosen so that ${}_2F_1(a,b,c;0)=1$.
\end{proof}
We stress that ${}_2F_1(a,b,c;x)$ solves Eq. \eqref{hyperEq} only whenever the restrictions imposed on the contour are met. Next step consist of increasing the number of these particular solutions by transforming the equation through means of the Nikiforov-Uvarov reduction explained throughout Section \ref{reduction}:
\begin{lemma}
\label{partSolsGauss}
Assume $w(x)=f(a,b,c;x)$ solves Eq. \eqref{hyperEq}. Then,
\begin{align}
&w_1(x) = x^{1-c}f(a-c+1,b-c+1,2-c;x), \label{sol2Gauss}\\
&w_2(x) = \left(1-x\right)^{c-a-b}f(c-a,c-b,c;x), \label{sol3Gauss} \\
&w_3(x) = f(b, a, c;x), \label{sol4Gauss}
\end{align}
also formally solve Eq. \eqref{hyperEq}.
\end{lemma}
\begin{proof}
Transform Eq. \eqref{hyperEq} following NU reduction and the identification $$\phi(x)=x(1-x), \quad \widetilde{\psi}(x)=c-\left(a+b+1\right)x, \quad \widetilde{\phi}(x)=-a b \phi(x).$$
First step, see Proposition \ref{P2lemma}, is to find $k_0 \in \mathbb{C}$ satisfying that $\Delta P_2(k_0)=0$ for
$$
P_2(x;k) = \left(\frac{1-c+\left(a+b-1\right)x}{2}\right)^{2}+(k+a b)x(1-x).
$$
Obtain two possible values, $k_1,\,k_2 \in \mathbb{C}$ for such $k_0 \in \mathbb{C}$,
$$
k_1=-c^2+\left(b+a+1\right)c-(a+1)b-a, \quad k_2=-ab.
$$
The transformation given by $k_0=k_2$ is redundant. Hence, set $k_0=k_1$, and obtain the following two possibilities for $\pi(x)$ and $\chi(x)$ given by Eqs. \eqref{pichoice} and \eqref{phi}, respectively,
\begin{enumerate}[]
\item[(A)] $\pi_1(x) = (1-c)(1-x), \quad \chi_1(x)=x^{1-c},$
\item[(B)] $\pi_2(x) = (a+b-c)x, \;\,\,\quad \chi_2(x)=(1-x)^{c-a-b}$.
\end{enumerate}
Accordingly, Eqs. \eqref{tau,sigmabar} and \eqref{lambda} yields 
\begin{enumerate}[]
\item[(A)] $\lambda_1 = -(a-c+1)(b-c+1), \quad \psi_1(x)=2-c-x\left(a+b+3-2c\right)$,
\item[(B)] $\lambda_2 = -(c-a)(c-b), \quad \quad \quad \quad \;\, \psi_2(x)=c-x\left(1-a-b+2c\right)$.
\end{enumerate}
Hence, cases (A) and (B) correspond to an Eq. \eqref{hyperEq} whose parameters $(a,b,c) \in \mathbb{C}^{3}$ are given by 
\begin{enumerate}
    \item[(A)] $a_1 = a -c + 1, \quad b_1=b-c+1, \quad c_1=2-c,$
    \item[(B)] $a_2 = c - a, \quad \quad \;\;\, b_2=c - b, \quad \quad \;\;\, c_2=c,$
\end{enumerate}
respectively.
Finally, solutions of the original GHE are written as $\chi_i(x)y_i(x)$ where $y_i(x)$ solves the GHDE defined by cases (A) or (B). Therefore,
\begin{align*}
&w_1(x) = x^{1-c}f(a-c+1,b-c+1,2-c;x),\\
&w_2(x) = \left(1-x\right)^{c-a-b}f(c-a,c-b,c;x), \\
&w_3(x) = f(b, a, c;x),
\end{align*}
solve Eq. \eqref{hyperEq} assuming ${}_2F_1(a,b,c;x)$ is a particular solution. As for $w_3(x)$, it is obtained by noticing that Eq. \eqref{hyperEq} is invariant by interchanging $a$ and $b$. Hence, the result.
\end{proof}
We obtained a particular solution for Eq. \eqref{hyperEq} by Lemma \ref{partSolGauss}. Now, introducing ${}_2F_1(a,b,c;x)$ in previous Lemma \ref{partSolsGauss}, we increase the number of linearly independent solutions in, at least, an specific region. Indeed, put $w(x) = {}_2F_1(a,b,c;x)$ and obtain associated $w_i(x)$ through Eqs. \eqref{sol2Gauss}, \eqref{sol3Gauss} and \eqref{sol4Gauss}. Their integral representation exist simultaneously provided that $$0 < \Re(a) < 1, \quad 0 < \Re(c - a) < 1.$$ In that case, associated $w(x)$ and $w_1(x)$ are linearly independent if we further assume $c \neq 1$: by definition, $w(0)={}_2F_1(a, b, c, 0) = 1,$ which, in turn, gives 
\begin{gather*}
    \lim_{x \to 0} w_1(x) = \lim_{x \to 0} x^{1-c}\pFq[skip=4]{2}{1}{a-c+1,b-c+1}{2-c}{x} = 0,\quad \Re(1-c)>0,
\end{gather*}
for instance. In addition, there exists a relation among the rest of the $w_{i}(x)$. By comparison, we find
\begin{align*}
  &\pFq[skip=4]{2}{1}{a,b}{c}{x}=\left(1-x\right)^{c-a-b}\pFq[skip=4]{2}{1}{c-a,c-b}{c}{x}=w_2(x), \\
  &\pFq[skip=4]{2}{1}{a,b}{c}{x} = \pFq[skip=4]{2}{1}{b,a}{c}{x}=w_3(x).
\end{align*}
Hence, we may replace representation of ${}_2F_1(a,b,c;x)$ given in Eq. \eqref{hypergeometricFunction} by the simpler
$$
\pFq[skip=4]{2}{1}{a,b}{c}{x} = \frac{\Gamma(c)}{\Gamma(a)\Gamma(c-a)}\int_{0}^{1}\dfrac{t^{a-1}\left(1-t\right)^{c-a-1}}{\left(1-xt\right)^{b}}\mathrm{d}t,
$$
provided that the adequate conditions on the parameters hold. This temporarily reduces the region of validity of the solution.
However, through next Lemma \ref{lemmaAnaHyper}, we show that such $F(a, b, c; x)$ is analytic in a domain bigger than the originally imposed:

\begin{lemma}
\label{lemmaAnaHyper}
The hypergeometric function, ${}_2F_1(a, b, c; z)$, defined by the integral representation
\begin{equation}
    \pFq[skip=4]{2}{1}{a,b}{c}{z} = \frac{\Gamma(c)}{\Gamma(a)\Gamma(c-a)}\int_{0}^{1}t^{a-1}\left(1-t\right)^{c-a-1}\left(1-zt\right)^{-b}\mathrm{d}t,
    \label{hypergeometricFunction2}
\end{equation}
is analytic in each $a$, $b$, $c$ and $z$ in $\mathbb{C}$ for $\Re(c)>\Re(a)>0$ and $\left|\arg(1-z)\right| < \pi$.
\end{lemma}
\begin{proof} 
The domain of analyticity is obtained by finding the domain in which the integral in Eq. \eqref{hypergeometricFunction2} converges uniformly with respect to $z$ and the rest of parameters $a$, $b$ and $c$, see Theorem \ref{theoA.1}. This occurs in the regions $$\left| z \right| \le N, \;  \left|\text{arg}(1-\delta-z)\right| \le \pi - \delta, \quad  \delta \le \Re(a) \le N, \quad \delta \le \Re(c-a) \le N, \quad \left|b\right| \le N, $$ for $\delta,\,N > 0$ because, whenever $0 \le t \le 1$, $$\left|t^{a-1}\left(1-t\right)^{c-a-1}\left(1-zt\right)^{-b}\right| < Ct^{\delta-1}\left(1-t\right)^{\delta-1},$$ and its integral converges. The restriction $\left|\text{arg}(1-\delta-z)\right| \le \pi - \delta$ removes the singular point $z=t^{-1}$. Setting $\delta \to 0$ and $N \to \infty$, the result follows. Last condition means that there is a cut in the $z$ plane along the real axis for $z \ge 1$.
\end{proof}
Now, we further continue analytically both ${}_2F_1(a, b, c; z)$ and its derivatives by finding the following recursion and differential relations:
\begin{lemma}
\label{hyperRecurLemma}
    The following relations hold:
    \begin{align}
&\hypergeometricsetup{
  symbol=\dfrac{\partial\, {}_2F_1}{\partial z},
}\pFq[skip=4]{}{}{a,b}{c}{z} = \dfrac{a b}{c}\hypergeometricsetup{
  symbol=F,
}\pFq[skip=4]{2}{1}{a+1,b+1}{c+1}{z}, \label{diffhyperRec} \\
&\hypergeometricsetup{
  symbol=\mathbf{F},
}\pFq[skip=4]{2}{1}{a,b}{c}{z}=(a+1)(b+1)z(1-z)\pFq[skip=4]{2}{1}{a+2,b+2}{c+2}{z}+\label{recurhyper} \\
&\hypergeometricsetup{
  symbol=\mathbf{F},
}\quad \quad \quad \quad \quad \quad \quad \quad \; \, \left(c-(a+b+1)z\right)\pFq[skip=4]{2}{1}{a+1,b+1}{c+1}{z}, \nonumber \\
&\hypergeometricsetup{
  symbol=\mathbf{F},
}\pFq[skip=4]{2}{1}{a,b}{c}{z}=(c-a+1)(c-b+1)\dfrac{z}{1-z}\pFq[skip=4]{2}{1}{a,b}{c+2}{z} +\label{recurhyper2} \\
&\hypergeometricsetup{
  symbol=\mathbf{F},
}\quad \quad \quad \quad \quad \quad \quad \quad \; \quad \dfrac{c-(2a-a-b+1)z}{1-z}\pFq[skip=4]{2}{1}{a,b}{c+1}{z}, \nonumber
\end{align}
where ${}_2F_1(a,b,c;z)$ is given by Eq. \eqref{hypergeometricFunction2}, and ${}_2\mathbf{F}_1(a,b,c;z)$ is defined by 
\begin{equation}
    \label{hyperTheta}
\hypergeometricsetup{
  symbol=\mathbf{F},
}\pFq[skip=4]{2}{1}{a,b}{c}{z}\coloneqq\dfrac{1}{\Gamma(c)}\hypergeometricsetup{
  symbol=F,
}\pFq[skip=4]{2}{1}{a,b}{c}{z}.
    \end{equation}
\end{lemma}
\begin{proof}
Introduce derivative with respect to $z$ in the integral representation given by Eq. \eqref{hypergeometricFunction2} by means of Theorem \ref{theoA.1}. Obtain
\begin{equation*}
    \hypergeometricsetup{
  symbol=\dfrac{\partial\, {}_2F_1}{\partial z},
}\pFq[skip=4]{}{}{a,b}{c}{z} = b\frac{\Gamma(c)}{\Gamma(a)\Gamma(c-a)}\int_{0}^{1}\dfrac{t^{a}\left(1-t\right)^{c-a-1}}{\left(1-zt\right)^{b+1}}\mathrm{d}t,
\end{equation*}
which corresponds to Eq. \eqref{diffhyperRec} since $\Gamma(z)=z\Gamma(z-1)$. As for Eq. \eqref{recurhyper}, introduce Eq. \eqref{diffhyperRec} into Eq. \eqref{hyperEq} which is valid by the principle of analytic continuation. Finally, replace $a$ and $b$ by $c-a$ and $c-b$ in Eq. \eqref{recurhyper} to obtain Eq. \eqref{recurhyper2}.
\end{proof}
Notice that corresponding relations given by Eqs. \eqref{recurhyper} and \eqref{recurhyper2} for ${}_2F_1(a,b,c;z)$ contain some terms multiplied by $\Gamma(c)$, $\Gamma(c+1)$ or $\Gamma(c+2)$, and, so, it is not possible to continue it analytically to the whole complex space $\mathbb{C}^4$ by these means. However, it is possible to continue ${}_2\mathbf{F}_1(a,b,c;z)$
given by Eq. \eqref{hyperTheta} to the whole complex space and, consequently, ${}_2F_1(a,b,c;z)$ with the further restriction $-c \notin \mathbb{N}_0$.

Indeed, by Lemma \ref{lemmaAnaHyper}, ${}_2F_1(a,b,c;z)$ is analytic in each $a$, $b$, $c$ and $z$ for $\Re(c)>\Re(a)>0$ and so is ${}_2\mathbf{F}_1(a,b,c;z)$. Now, whenever $\Re(c) \le \Re(a)$, we increase repeatedly $c$ by making use of Eq. \eqref{recurhyper2}. Whenever $\Re(a) \le 0$, we increase repeatedly $a$ and $c$ by Eq. \eqref{recurhyper}. Hence, we obtain an analytic continuation of ${}_2\mathbf{F}_1(a,b,c;z)$ to the whole complex space. Finally, Eq. \eqref{diffhyperRec} continues its derivatives and, by the principle of analytic continuation and Lemma \ref{partSolGauss}, we obtain a solution of Eq. \eqref{hyperEq} valid for every value of its parameters.

Putting it all together, we obtain the following result:
\begin{prop}
\label{gaussProp}
    The hypergeometric functions
    \begin{equation}
        \begin{split}
            &y_1(z)=\pFq[skip=4]{2}{1}{a,b}{c}{z}, \\
            &y_2(z)=z^{1-c}\pFq[skip=4]{2}{1}{a-c+1,b-c+1}{2-c}{z},
        \end{split}
        \label{solHyper}
    \end{equation}
    with integral representation given by Eq. \eqref{hypergeometricFunction2}, are analytic in each variable and are linearly independent solutions of Eq. \eqref{hyperEq} under the restriction $c \notin \mathbb{Z}$.
\end{prop}
\begin{proof}
    It collects the result obtained throughout this epigraph.
\end{proof}
When $c \in \mathbb{Z}$, either $y_1(z)$ or $y_2(z)$ in Eq. \eqref{solHyper} becomes indeterminate and the problem of finding linearly independent solutions persists. That issue is addressed in \cite[\S21.4, pp. 277--281]{Nikiforov1988SpecialFO} and reader is referred there for a complete discussion. Here, we simply recollect the linearly independent solutions of Eq. \eqref{hyperEq} for any possible values of the parameters in Table \ref{linIndHyper}. There, the so-called Pochhammer symbols is defined as \begin{equation}
\left(a\right)_k\coloneqq a(a+1)\dots(a+k-1),\quad \left(a\right)_0\coloneqq 1,
    \label{psymbol}
\end{equation} whereas
\begin{equation}
a'\coloneqq a-c+1,\quad b'\coloneqq b-c+1, \quad c'\coloneqq 2-c.
\label{primeParameters}
\end{equation}
Finally, ${}_2\Phi_1(a,b,c;z)$ is defined as
\begin{equation}
\begin{split}
    \hypergeometricsetup{
  symbol=\Phi,
}\pFq[skip=2]{2}{1}{a,b}{c}{z} \coloneqq & \hypergeometricsetup{
  symbol=\dfrac{\partial\, {}_2F_1}{\partial c},
}\pFq[skip=2]{}{}{a,b}{c}{z}\\
-&\dfrac{\Gamma(a')\Gamma(b')\Gamma(c)}{\Gamma(a)\Gamma(b)}\dfrac{\partial}{\partial c}\left(\dfrac{z^{1-c}}{\Gamma(c')}\hypergeometricsetup{
  symbol=F,
}\pFq[skip=2]{2}{1}{a',b'}{c'}{z}\right).
\end{split}
\label{hyperFunc3}
\end{equation}

\begin{table}[htb!]
\centering
\begin{tabular}{c c c c}
\toprule
$c$ & $a, b$ & $y_1(z)$ & $y_2(z)$ \\ [0.5ex]  \midrule 
$c \notin \mathbb{Z}$ & $a, b$ arbitrary & ${}_2F_1(a,b,c;z)$ & $z^{1-c}{}_2F_1(a',b',c';z)$\\ [0.5ex]  \midrule
\multirow{2}{*}{$c-1=m\in\mathbb{N}_0$} & $\left(a'\right)_m\left(b'\right)_m=0$ & ${}_2F_1(a,b,c;z)$ & $z^{1-c}{}_2F_1(a',b',c';z)$ \\ 
 & $\left(a'\right)_m\left(b'\right)_m\neq0$ & ${}_2F_1(a,b,c;z)$ & ${}_2\Phi_1(a,b,c;z)$ \\ [0.5ex] \midrule
\multirow{2}{*}{$c'-1=m \in \mathbb{N}$} & $\left(a\right)_m\left(b\right)_m=0$ & ${}_2F_1(a,b,c;z)$ & $z^{1-c}{}_2F_1(a',b',c';z)$ \\ 
 & $\left(a\right)_m\left(b\right)_m\neq0$ & $z^{1-c}{}_2\Phi_1(a',b',c';z)$ & $z^{1-c}{}_2F_1(a',b',c';z)$ \\ [0.5ex]
 \bottomrule
\end{tabular}
\caption{Linearly independent solutions, $y_1(z)$ and $y_2(z)$, of Eq. \eqref{hyperEq} for different parameter values. Function ${}_2F_1(a,b,c;z)$ is given by Eq. \eqref{hypergeometricFunction2} and ${}_2\Phi_1(a,b,c;z)$ by Eq. \eqref{hyperFunc3}. Symbol $(a)_m$ is defined in Eq. \eqref{psymbol} and $a',b',c'$ in Eq. \eqref{primeParameters}.}
\label{linIndHyper}
\end{table}

We offer a few remarks. It is possible to construct linearly independent solutions using those obtained particularly in Lemma \ref{partSolGauss}. However, it is more convenient, in this context, to use ${}_2\Phi_1(a,b,c;z)$. Notice, finally, the symmetry in $(a,b,c)$ and $(a',b',c')$: whenever $c-1=m\in \mathbb{N}_0$, $c(c')$ is undefined and ${}_2F_1(a',b',c';z)$ may not exist. In certain occasions, namely, whenever $\left(a'\right)_m\left(b'\right)_m=0$, this indeterminacy is conveniently solved. However, whenever $\left(a'\right)_m\left(b'\right)_m \neq 0$, the indeterminacy must be removed. In that process,  ${}_2\Phi_1(a,b,c;z)$ appears. Whenever $c'-1=m\in \mathbb{N}$, the roles are reversed.
\subsubsection{Confluent hypergeometric equation\label{confluent}}
Whenever $\phi(x) \in \mathcal{P}_1,$ Eq. \eqref{HDE} can be expressed as a confluent hypergeometric equation:
\begin{definition}
Let $I \subset \mathbb{R}$ be an open interval not necessarily bounded. A confluent hypergeometric equation or CHE is a differential equation of the form
\begin{equation} 
x\dfrac{\mathrm{d}^2y}{\mathrm{d}x^2}(x)+\left(c-x\right)\dfrac{\mathrm{d}y}{\mathrm{d}x}(x)-a y(x)=0, \quad x \in I,
\label{confluentHyperEq}
\end{equation}
where $a, c \in \mathbb{C}$ are arbitrary. In addition, $\omega(x)$ and $\nu \in \mathbb{C}$ for its associated Eqs. \eqref{pearson} and \eqref{lambdanu} are given by
\begin{equation}
        \omega(x)=x^{c-1}e^{-x}, \quad \nu = -a,
        \label{confluentPar}
\end{equation}
respectively.
\end{definition}
Again, we obtain particular solutions by means of Lemma \ref{lemmaParticularSolutions}:
\begin{lemma}
\label{partSolConfluent}
Consider Eq. \eqref{confluentHyperEq} for $I=(0,\infty)$. Then, the following two contours $u_i:J \to \mathbb{C}$ for $i=1,2$ satisfy Condition \ref{firstCond} associated to such CHE under the parametric restrictions specified:
\begin{align*}
    & u_1(t)=xt, & \Re(c)>\Re(a)>2, & & J = [0,1], \\
    & u_2(t)=x(1+t), & \Re(a)>2, & & J =  [0,\infty).
\end{align*}
Their associated particular solutions, $y_i(z)$, are written
\begin{align*}
&y_1(x) = \pFq[skip=0]{1}{1}{a}{c}{x},\\
&y_2(x) = \hypergeometricsetup{symbol=G,}\pFq[skip=4]{1}{1}{a}{c}{x},
\end{align*}
where\footnote{Again, we also write ${}_1F_1(a,c;x)$ or ${}_1G_1(a,c;x)$ whenever needed, accordingly.}
\begin{align}
    &\hypergeometricsetup{symbol=F}\pFq[skip=4]{1}{1}{a}{c}{x}=\frac{\Gamma(c)}{\Gamma(a)\Gamma(c-a)}e^x \int^{1}_{0}\dfrac{t^{c-a-1}\left(1-t\right)^{a-1}}{e^{xt}}\mathrm{d}t, \quad \Re(c)>\Re(a),
    \label{confluentHyper} \\
    &\hypergeometricsetup{symbol=G,}\pFq[skip=4]{1}{1}{a}{c}{x}=\frac{1}{\Gamma(a)} \int^{\infty}_{0}\dfrac{t^{a-1}\left(1+t\right)^{c-a-1}}{e^{xt}}\mathrm{d}t,
    \label{confluentHyper2}
\end{align}
whenever $\,0 < x < \infty$ and $\Re(a)>2$. They are the so-called confluent hypergeometric function of first and second kind, respectively.
\end{lemma}
\begin{proof}
For Eq. \eqref{confluentHyperEq}, Condition \ref{firstCond} translates to
$$
p(u_2)-p(u_1) \coloneqq s^{c-a}e^{-u}\left(u-x\right)^{a-2}\Bigr|^{u_2}_{u_1}=0,
$$
where $\omega(x)$ and $\nu \in \mathbb{C}$ are given by Eq. \eqref{confluentPar}. Under such restrictions, both $p(u_1)$ and $p(u_2)$ vanish for the two contours. Define
$$
f(t,x) = t^{c-a-1}\left(1-t\right)^{a-1}e^{-xt}, \quad (t,x) \in (0,1) \times (0,\infty) \coloneqq \Omega.
$$
Hence, $f(t,x) \in C(\Omega)$ and $f(t_0,x) \in \mathcal{H}(0,\infty)$ for every $t_0 \in (0,1)$ and Theorem \ref{theoA.1} is employed for Condition \ref{secondCond}.

Accordingly to Lemma \ref{particularSolutionTheo}, ${}_1F_1(a,c;x)$ given by Eq. \eqref{confluentHyper} constitutes a particular solution of Eq. \eqref{confluentHyperEq} for $I=(0,\infty)$ whenever $\Re(c)>\Re(a)>2$. Solution $G(a,c;z)$ is obtained similarly. 

Constants are chosen so that ${}_1F_1(a,c;0)=1$ and $\lim_{x \to \infty} x^{a}{}_1G_1(a,c;x)=1$.
\end{proof}

Now, by means of the NU method, we increase the number of particular solutions obtained in previous Lemma \ref{partSolConfluent}:
\begin{lemma}
\label{lemmaConfluentInc}
Assume $w(x) = f(a,c;x)$ is a solution of Eq. \eqref{confluentHyperEq}. Then,
\begin{align*}
&w_1(x) = x^{1-c}f(a-c+1,2-c;x),\\
&w_2(x) = e^{x}f(c-a,c;-x),
\end{align*}
solve Eq. \eqref{confluentHyperEq}.
\end{lemma}
\begin{proof}
Transform Eq. \eqref{confluentHyperEq} into another HDE by following the NU reduction discussed in Section \ref{reduction}. Eq. \eqref{confluentHyperEq} is a GHE with polynomials $$\phi(x)=x, \quad \widetilde{\psi}(x)=c-x,\quad \widetilde{\phi}(x)=-a x,$$
compare with Eq. \eqref{GHE}. Firstly, find the following two possibilities for $\pi(x)$ and $\chi(x)$ given by Eqs. \eqref{piEq} and \eqref{phi} 
\begin{enumerate}[]
    \item[(A)] $\pi_1(x)=1-c, \quad \chi_1(x)=x^{1-c},$
    \item[(B)] $\pi_2(x)=x, \quad \quad \;\;\, \chi_2(x)=e^x.$
\end{enumerate}
Following Eqs. \eqref{tau,sigmabar} and \eqref{lambda}, write
\begin{enumerate}[]
    \item[(A)] $\psi_1(x)=(2-c-x), \quad \lambda_1=c-a-1,$
    \item[(B)] $\psi_2(x)=(x+c), \quad \quad \;\;\, \lambda_2=c-a.$
\end{enumerate}
Clearly, HDE of case (A) can be compared directly to Eq. \eqref{confluentHyperEq}. Thence, $v_1(x) = f(a-c+1,2-c;x)$ solves it and $w_1(x)=x^{1-c}v_1(x)$ the original CHE. As for the second HDE, set $\overline{x}=-x$ and rewrite
\begin{enumerate}[]
    \item[(B)] $\psi_2(\overline{x})=(c-\overline{x}), \quad \lambda_2=a-c.$
\end{enumerate}
Thence, $v_2(\overline{x})=f(c-a,c;\overline{x})$ solves it and $w_2(x) = e^{x}v_2(\overline{x})$ the original CHE. Hence, the result. 
\end{proof}
Now, we continue analytically both solutions ${}_1F_1(a,c;x)$ and ${}_1G_1(a,c;x)$ given by Eqs. \eqref{confluentHyper} and \eqref{confluentHyper2}:
\begin{lemma}
\label{anaConfluentLemma}
    Confluent hypergeometric function ${}_1F_1(a,c;z)$ defined by Eq. \eqref{confluentHyper} is analytic in each variable for $\Re(c)>\Re(a)>0$ and all $z \in \mathbb{C}$. 
    
    As for ${}_1G_1(a,c;z)$ defined by Eq. \eqref{confluentHyper2}, it is analytic for $\left|\text{arg}\,z\right|<3\pi/2,\, z \neq 0$ and $\Re(a)>0$.
\end{lemma}
\begin{proof}
    Again, find the region in which the integral
    $$I(z,a,c) \coloneqq \int^{1}_{0} t^{c-a-1}\left(1-t\right)^{a-1}e^{-zt}\mathrm{d}t,$$
    converges uniformly in $z$, $a$ and $c$. This occurs in the regions $$\Re(z) \ge -N,\quad  \delta \le \Re(a) \le N, \quad \delta \le \Re(c-a) \le N,$$ for $\delta,\,N > 0$ because, whenever $0 \le t \le 1$, $$\left|t^{c-a-1}\left(1-t\right)^{a-1}e^{-zt}\right| < t^{\delta-1}\left(1-t\right)^{\delta-1}e^{Nt},$$ and its integral converges. Hence, the first part follows as $N$ and $\delta$ are arbitrary.
    
    The integral defining ${}_1G_1(a,c;x)$ is a Laplace integral. From the discussion in \cite[Appendix B, Theorem 1, p.383]{Nikiforov1988SpecialFO}, the result follows.
\end{proof}
To make ${}_1G_1(a,c;z)$ single-valued it is enough to introduce a cut along the real axis for $z<0$ and assume $-\pi < \text{arg}\,z < \pi.$ Next, we obtain differential and recursion relations that, together with previous Lemma \ref{anaConfluentLemma}, will allow us to continue further both confluent hypergeometric functions:
\begin{lemma}
\label{confluentRecurLemma}
    The following relations hold:
    \begin{align}
        &\hypergeometricsetup{symbol=\dfrac{\partial\,{}_1F_1}{\partial z}}\pFq[skip=4]{}{}{a}{c}{z} = \dfrac{a}{c}\hypergeometricsetup{symbol=F}\pFq[skip=4]{1}{1}{a+1}{c+1}{x}, \label{diffcFnu} \\
        &\hypergeometricsetup{symbol=\mathbf{F}}\pFq[skip=4]{1}{1}{a}{c}{z}=(a+1)z\,\hypergeometricsetup{symbol=\mathbf{F}}\pFq[skip=4]{1}{1}{a+2}{c+2}{z}+(c-z)\hypergeometricsetup{symbol=\mathbf{F}}\pFq[skip=4]{1}{1}{a+1}{c+1}{z}, \label{recurcFnu} \\
        &\hypergeometricsetup{symbol=\mathbf{F}}\pFq[skip=4]{1}{1}{a}{c}{z}=(a-c-1)z\,\hypergeometricsetup{symbol=\mathbf{F}}\pFq[skip=4]{1}{1}{a}{c+2}{z}+(c+z)\hypergeometricsetup{symbol=\mathbf{F}}\pFq[skip=4]{1}{1}{a}{c+1}{z}, \label{recurcFnu2} \\
        &\hypergeometricsetup{symbol=\dfrac{\partial\, {}_1G_1}{\partial z}}\pFq[skip=4]{}{}{a}{c}{z} = -a \hypergeometricsetup{symbol=G}\pFq[skip=4]{1}{1}{a+1}{c+1}{z}, \label{diffcF2nu} \\
        &\hypergeometricsetup{symbol=G}\pFq[skip=4]{1}{1}{a}{c}{z}=(a+1)z\,\hypergeometricsetup{symbol=G}\pFq[skip=4]{1}{1}{a+2}{c+2}{z}-(c-z)\hypergeometricsetup{symbol=G}\pFq[skip=4]{1}{1}{a+1}{c+1}{z}, \label{recurcF2nu}
    \end{align}
    where ${}_1F_1(a,c;z)$ and ${}_1G_1(a,c;z)$ are given by Eqs. \eqref{confluentHyper} and \eqref{confluentHyper2}; and ${}_1\mathbf{F}_1(a,c;z)$ is defined by 
    \begin{equation}
\hypergeometricsetup{symbol=\mathbf{F}}\pFq[skip=4]{1}{1}{a}{c}{z}\coloneqq\dfrac{1}{\Gamma(c)}\hypergeometricsetup{symbol=F}\pFq[skip=4]{1}{1}{a}{c}{z}.
    \end{equation}
\end{lemma}
\begin{proof}
    Follow similar arguments that those employed for hypergeometric function ${}_2F_1(a,b,c;x)$ in Lemma \ref{hyperRecurLemma}.
\end{proof}

Putting together Lemmas \ref{anaConfluentLemma} and \ref{confluentRecurLemma}, we continue analytically both ${}_1\mathbf{F}_1(a,c;z)$ and ${}_1G_1(a,c;z)$ to $a$, $c$ and $z$ to $\mathbb{C}$. Subsequently, we extend ${}_1F_1(a, c;z)$ with the further restriction $-c \notin \mathbb{N}_0$. In addition, Eqs. \eqref{diffcFnu} and \eqref{diffcF2nu} yields that both ${}_1\mathbf{F}_1(a,c;z)$ and ${}_1G_1(a,c;z)$ solves Eq. \eqref{confluentHyperEq} without restrictions in the parameter. Same considerations apply to ${}_1F_1(a,c;z)$ whenever $-c \notin \mathbb{N}_0$.

Finally, we construct linearly independent solutions:
\begin{prop}
\label{conlfuentProp}
    The confluent hypergeometric functions
    \begin{equation}
        \begin{split}
&y_1(z)=\hypergeometricsetup{symbol=F}\pFq[skip=4]{1}{1}{a}{c}{z}, \\
            &y_2(z)=z^{1-c}\hypergeometricsetup{symbol=F}\pFq[skip=4]{1}{1}{a-c+1}{2-c}{z},
        \end{split}
        \label{solConfluent}
    \end{equation}
    with integral representation given by Eq. \eqref{confluentHyper}, are analytic in each variable and are linearly independent solutions of Eq. \eqref{confluentHyperEq} under the restriction $c \notin \mathbb{Z}$. In addition, the following functional equation is valid
    \begin{equation*}
        \hypergeometricsetup{symbol=F}\pFq[skip=4]{1}{1}{a}{c}{x}=e^{z}\hypergeometricsetup{symbol=F}\pFq[skip=4]{1}{1}{c-a}{c}{-z}.
    \end{equation*}
\end{prop}
\begin{proof}
    Again, arguments are similar to those employed in Subsection \ref{gauss}.
\end{proof}
For the confluent hypergeometric equation, we obtain no restrictions:
\begin{prop}
\label{conlfuentProp2}
    The confluent hypergeometric functions
    \begin{equation}
        \begin{split}
            &y_1(z)=\hypergeometricsetup{symbol=G}\pFq[skip=4]{1}{1}{a}{c}{z}, \\
            &y_2(z)=e^{z}\hypergeometricsetup{symbol=G}\pFq[skip=4]{1}{1}{c-a}{c}{-z},
        \end{split}
        \label{solConfluent2}
    \end{equation}
    with integral representation given by Eq. \eqref{confluentHyper2}, are analytic in each variable and are linearly independent solutions of Eq. \eqref{confluentHyperEq}. In addition, the following functional equation is valid
    \begin{equation*}
        \hypergeometricsetup{symbol=G}\pFq[skip=4]{1}{1}{a}{c}{z}=z^{1-c}\hypergeometricsetup{symbol=G}\pFq[skip=4]{1}{1}{a-c+1}{2-c}{-z}.
    \end{equation*}
\end{prop}
Depending on the context, we work with ${}_1F_1(a,c;z)$ or ${}_1G_1(a,c;z)$.

\subsubsection{Hermite equation\label{hermite}}
Whenever $\phi(x)\in \mathcal{P}_0,$ Eq. \eqref{HDE} can be expressed as a Hermite equation:
\begin{definition}
Let $I \subset \mathbb{R}$ be an open interval not necessarily bounded. A Hermite equation or HE is a differential equation of the form
\begin{equation} 
\dfrac{\mathrm{d}^2y}{\mathrm{d}x^2}(x)-2x\dfrac{\mathrm{d}y}{\mathrm{d}x}(x)+2\nu y(x)=0, \quad x \in I,
\label{hermiteEq}
\end{equation}
where $\nu \in \mathbb{C}$ is arbitrary. In addition, $\omega(x)$ for its associated Eq \eqref{pearson} is given by
\begin{equation*}
\omega(x)=e^{-x^2}.
\label{hermitePar}
\end{equation*}
\end{definition}
We construct particular solutions by means of Lemma \ref{lemmaParticularSolutions}:
\begin{lemma}
\label{partSolHermite}
Consider Eq. \eqref{hermiteEq} for $I=(0,\infty)$. Then, contour $u_1(t)=x+t$ with $t \in [0,\infty)$ under restrictions $\Re(\nu)<-2$ satisfies Condition \ref{firstCond} associated to such HE. Its associated solution, $y_1(x)$, is written $y_1(x) = H_{\nu}(x)$ where
\begin{equation}
    H_{\nu}(x) \coloneqq \frac{1}{\Gamma(-\nu)} \int^{\infty}_{0}e^{-t^2-2xt}t^{-\nu-1}\mathrm{d}t, \quad x \in (0,\infty), \quad \Re(\nu)<-2,
    \label{hermiteFunc}
\end{equation}
is the so-called Hermite function.
\end{lemma}
\begin{proof}
For Eq. \eqref{hermiteEq}, Condition \ref{firstCond} translates to
$$
p(u_2)-p(u_1) \coloneqq e^{-u^2}\left(u-x\right)^{-\nu-2}\Bigr|^{u_2}_{u_1}=0,
$$
where $\omega(x)$ and $\nu \in \mathbb{C}$ are given by Eq. \eqref{hermitePar}. Under such restrictions, both $p(u_1)$ and $p(u_2)$ vanish for the contour. Define
$$
f(t,x) = t^{-t^2-2xt}t^{-\nu-1}, \quad (t,x) \in (0,\infty) \times (0,\infty) \coloneqq \Omega.
$$
Hence, $f(t,x) \in C(\Omega)$ and $f(t_0,x) \in \mathcal{H}(0,\infty)$ for every $t_0 \in (0,\infty)$ and Condition \ref{secondCond} holds by Theorem \ref{theoA.1}.

Accordingly to Lemma \ref{particularSolutionTheo}, $H_{\nu}(x)$ given by Eq. \eqref{hermiteFunc} constitutes a particular solution of Eq. \eqref{hermiteEq} for $I=(0,\infty)$ whenever $\Re(a)<-2$.

Finally, the constant is chosen so that $H_{n}(x)$ for $n=0,1,\dots$ coincides with the Hermite polynomials.
\end{proof}

Next, we increase the number of particular solution:
\begin{lemma}
\label{lemmaHermiteInc}
Assume $w(x) = f_{\nu}(x)$ is a solution of Eq. \eqref{hermiteEq}. Then,
\begin{align*}
&w_1(x) = e^{-x^2}f_{-\nu-1}(ix),\\
&w_2(x) = f_{\nu}(-x), \\
&w_3(x) = e^{-x^2}f_{-\nu-1}(-ix),
\end{align*}
solve Eq. \eqref{hermiteEq}.
\end{lemma}
\begin{proof}
    Invariance of Eq. \eqref{hermiteEq} under replacement of $x$ by $-x$ gives solutions $w_2(x)$ and $w_3(x)$. As for $w_1(x)$, check it solves Eq. \eqref{hermiteEq} by direct computations.
\end{proof}

By setting $w(x)=H_{\nu}(x)$ defined by Eq. \eqref{hermiteFunc} in Lemma \ref{lemmaHermiteInc}, it is straightforward to check that $w(x)$ and $w_2(x)$ behave differently as $x \to \infty$ for $\Re(\nu)<-2$ provided that we extend further the domain of analyticity to $z \in \mathbb{R}$. Actually, we extend it to $z \in \mathbb{C}$:
\begin{lemma}
\label{anaHermiteLemma}
    Hermite function $H_{\nu}(z)$ defined by Eq. \eqref{hermiteFunc} is analytic in each variable for $\Re(\nu)<0$.
\end{lemma}
\begin{proof}
Find the region in which the integral
$$I(z,\nu) \coloneqq \int^{\infty}_{0}e^{-t^2-2zt}t^{-\nu-1}\mathrm{d}t,$$
converges uniformly in $z$ and $\nu$. This occurs in the regions $$\Re(z) \ge -N, \delta - 1 \le -\Re(\nu)-1 \le N,$$ for $N,\delta > 0$ because $$\left|e^{-t^2-2zt}t^{-\nu-1}\right| < e^{-t^2+2Nt}\left(t^{\delta-1}+t^{N}\right),$$ and its integral converges. Hence, the result as $N$ and $\delta$ are arbitrary.
\end{proof}

Finally, we obtain differential and recursion relations:
\begin{lemma}
    The following relations hold
    \begin{align}
        &\dfrac{\partial}{\partial z}H_{\nu}(z) = 2\nu H_{\nu-1}(z), \label{diffHnu} \\
        &H_{\nu}(z)=2zH_{\nu-1}(x)-(2\nu-2)H_{\nu-2}(z), \label{recurHnu}
    \end{align}
    for $H_{\nu}(x)$ define in Eq. \eqref{hermiteFunc}.
\end{lemma}
\begin{proof}
    Eq. \eqref{diffHnu} is obtained in a straightforward manner. As for Eq. \eqref{recurHnu}, combine Eq. \eqref{hermiteEq} with Eq. \eqref{diffHnu} for $y(z)=H_{\nu}(z)$.
\end{proof}

Putting it all together, we solve Eq. \eqref{hermiteEq}:
\begin{prop}
\label{hermiteProp}
    The Hermite functions $H_{\nu}(z)$ and $H_{\nu}(-z)$ with integral representation given by Eq. \eqref{hermiteFunc} are analytic in each variable $z \in \mathbb{C}$ and $\nu \in \mathbb{C}$, linearly independent and solve Eq. \eqref{hermiteEq} without any restrictions on the parameter $\nu \in \mathbb{C}$.
\end{prop}
\begin{proof}
    By Lemma \ref{anaHermiteLemma} and Eq. \eqref{recurHnu}, $H_{\nu}(z)$ and $H_{\nu}(-z)$ are continued analytically to $z,\nu \in \mathbb{C}$. By Eq. \eqref{diffHnu}, this analytic continuation also applies to derivatives of any order. Hence, by the principle of analytic continuation and Lemma \ref{partSolHermite}, $H_{\nu}(z)$ and $H_{\nu}(-z)$ solve Eq. \eqref{hermiteEq}. Finally, they clearly behave differently as $z \to \infty$ and, so, are linearly independent.
\end{proof}

\subsubsection{Addressing the scattering states\label{scattetingstates}}
Recall the form of Problem \ref{schroProb}. Ideally, the region \ref{R1}, associated with the bound states, is solved using Theorem \ref{eigenTheo} in terms of Classical Orthogonal Polynomials (OPs). For region \ref{R2}, the approach is straightforward: after a suitable NU reduction given by $\Psi_{\varepsilon}(x)=\chi(x;\varepsilon) y_{\varepsilon}(x)$, express the resulting HDE as a GHDE, CHE, or HE, and solve it using Propositions \ref{gaussProp}, \ref{conlfuentProp}, \ref{conlfuentProp2}, and \ref{hermiteProp}. Subsequently, check for boundedness by utilizing specific properties of the special functions ${}_2F_1(a,b,c;z)$, ${}_1F_1(a,c;z)$, ${}_1G_1(a,c;z)$ and $H_{\nu}(z)$ obtained earlier. For conciseness, these results are presented without proofs in Appendix \ref{basicProp}: we depart from our primary reference \cite{Nikiforov1988SpecialFO} and supplement it by the more formulaic and convenient approach found in NIST's Digital Library of Mathematical Functions\footnote{Available online at \url{https://dlmf.nist.gov}} \cite{NIST}. In particular, we focus on its Chapters $\S13.2$, $\S15.4(ii)$ and $\S15.10$ where Asymptotic Behaviours, Wronskians and Connection Formulas for ${}_1F_1(a,c;z)$, ${}_1\mathbf{F}_1(a,c;z)$,
${}_1G_1(a,c;z)$, ${}_2F_1(a,b,c;z)$ and ${}_2\mathbf{F}_1(a,b,c;z)$ are presented. Thus, we compare their asymptotics with that of the particle $\chi(x;\varepsilon)$ obtained via the NU reduction. For completeness, we present some of these properties in Appendix \ref{basicProp}.

We offer two further remarks. The results are expressed in terms of the hypergeometric series ${}_pF_q(a_1,\dots,a_p,b_1,\dots,b_q;z)$ generally defined as
\begin{equation}
    \label{hyperSeries}
    \hypergeometricsetup{symbol=F, fences=parens}\pFq[skip=4]{p}{q}{a_1,\dots,a_p}{b_1,\dots,b_q}{z} \coloneqq \sum^{\infty}_{n=0}\dfrac{(a_1)_n\dots(a_p)_n}{(b_1)_n\dots(b_q)_n}\dfrac{z^n}{n!}.
\end{equation}
The series can converge only whenever $p \le q+1$; for $p=q+1$, it converges only when $|z|<1$. In their domain of convergence, they are finite and they constitute a representation for ${}_2F_1(a,b,c;z)$ and ${}_1F_1(a,c;z)$ defined in Eqs. \eqref{hypergeometricFunction2} and \eqref{confluentHyper}, i.e,
\begin{equation}
    \begin{split}
    &\hypergeometricsetup{symbol=F, fences=parens}\pFq[skip=4]{1}{1}{a}{c}{z}=\hypergeometricsetup{symbol=F, fences=brack}\pFq[skip=4]{1}{1}{a}{c}{z}, \\
    &\hypergeometricsetup{symbol=F, fences=parens}\pFq[skip=4]{2}{1}{a,b}{c}{z}=\hypergeometricsetup{symbol=F, fences=brack}\pFq[skip=4]{2}{1}{a,b}{c}{z}, \quad |z| \le 1,
\end{split}
\end{equation}
whenever $-c \notin \mathbb{N}_0$. Similarly, for instance,
\begin{equation}
    \begin{split}
    &\hypergeometricsetup{symbol=\mathbf{F}, fences=brack}\pFq[skip=4]{1}{1}{a}{c}{z}=\sum^{\infty}_{n=0}\dfrac{(a)_n}{\Gamma(c+n)}\dfrac{z^n}{n!}, \\
    &\hypergeometricsetup{symbol=\mathbf{F}, fences=brack}\pFq[skip=4]{2}{1}{a,b}{c}{z} =\sum^{\infty}_{n=0}\dfrac{(a)_n(b)_n}{\Gamma(c+n)}\dfrac{z^n}{n!}, \quad |z| \le 1.
\end{split}
\end{equation}
There also exists series representations for ${}_2\Phi_1(a,b,c;z)$ and ${}_1G_1(a,c;z)$ previously defined in Eqs. \eqref{hyperFunc3} and \eqref{confluentHyper2}. We properly introduce them when needed in following Section \ref{applications}.

Additionally, note that in \cite{NIST}, there is a shift in notation compared to the one used here. Specifically, their functions $M(a,c;z)$, $\mathbf{M}(a,c;z)$, $U(a,c;z)$ and $F(a,b,c;z)$ correspond to ${}_1F_1(a,c;z)$, ${}_1\mathbf{F}_1(a,c;z)$,
${}_1G_1(a,c;z)$ and ${}_2F_1(a,b,c;z)$, respectively.

\section{Three potentials solved by the NU method \label{applications}}
With the theory developed in the previous section, we can now apply it to three specific potentials: the Quantum Harmonic Oscillator, the Morse potential, and the Rosen-Morse II potential. These potentials are of particular interest in the field of supersymmetric quantum mechanics (SUSY) and represent three of the six known one-dimensional shape-invariant potentials, where certain parameters are related by a translation (see \cite[Table 4.1, pp. 296–297]{susy}). While the bound states of these potentials are typically solved using SUSY methods, our aim here is to demonstrate how our method works, and, where possible, to characterize the scattering states.

Additionally, we present Tables \ref{table:oscillator}, \ref{table:morse}, and \ref{table:rosenmorseII}, which contain all the results and necessary parameters required to apply the method to these potentials, highlighting the systematic approach this method provides for solving them. This demonstrates that the method is well-suited for a unified approach using computer algebra systems. In this context, we refer the reader to \cite{lineellis}, where the authors discuss basic (and well-known) potentials in both non-relativistic and relativistic quantum mechanics that can be integrated into the Nikiforov-Uvarov framework with the assistance of a computer algebra system. While their approach is more computational than ours, it further illustrates the power of the method.

Finally, unless stated otherwise, physical constants appearing in the systems collected are to be considered positive. As a rule of thumb, variable `$x$' will correspond to the original SE, variable `$\overline{x}$' to the dimensionless SE, variable `$s$' to the associated GHE/HDE and variable `$t$' to the associated GHDE, whenever needed.

\subsection{Quantum Harmonic Oscillator}
Our first example corresponds to one of the simplest quantum physical system: the onedimensional harmonic oscillator, which serves as the classical example for demonstrating the method. 

Any arbitrary smooth potential can usually be approximated as a harmonic potential at the vicinity of a stable point. It plays an important role in the foundations of quantum electrodynamics, and has applications to oscillations in crystal and molecules. Hence, its importance. 

Its (time-independent) Schr\"odinger equation is written as
\begin{equation}
\label{harmonicoscillatorreal}
    -\dfrac{\hbar^2}{2m}\Psi''(x)+\dfrac{1}{2}m\Omega^2x^2\Psi(x)=E\Psi(x), \quad x \in \mathbb{R}.
\end{equation}
Its potential $V(x)=\frac{1}{2}m\Omega^2x^2,$ ($m$ is the mass, $x$ represents the displacement from equilibrium and $\Omega$ is the angular frequency) satisfies $$V_{\text{min}}=0, \quad V_{\pm}=V_{\text{max}}=\infty,$$ remember Eq. \eqref{v+-}. Every solution must be a bound state $(E,\Psi(x;E))$ such that $E$ is strictly positive and $\Psi(x) \in L^2(\mathbb{R})$: region \ref{R2} is lost and there are no scattering states. By setting $$x_0=\sqrt{\dfrac{\hbar}{m\Omega}}, \quad x=x_0s, \quad E = \dfrac{\hbar\Omega\varepsilon}{2},$$ it is reduced to
\begin{equation}
    \dfrac{\mathrm{d^2}\Psi(s)}{\mathrm{d}s^2}+\left(\varepsilon-s^2\right)\Psi(s)=0, \quad s \in \mathbb{R}.
    \label{harmonicoscillator}
\end{equation}

\begin{table}[htpb!]
\centering
\begingroup
\renewcommand{\arraystretch}{1}
\begin{tabular}{r c}
\multirow{1}{*}{Harmonic Oscillator?}  & $-\dfrac{\hbar^2}{2m}\dfrac{\mathrm{d^2}\Psi(x)}{\mathrm{d}x^2}+\dfrac{1}{2}m\Omega^2x^2\Psi(x)=E\Psi(x).$  \\ 
\midrule
\multirow{3}{*}{Potential, $V(x)$?}  & $V_{\text{min}}=0,$ \\ 
                    & $V_{\pm}=\infty,$  \\ 
                    & $V_{\text{min}}=\infty.$  \\ 
\midrule
\multirow{2}{*}{Energy Regions?} & Bound states: $E \in (0,\infty)$.  \\ 
                                & No scattering states.  \\
\midrule
\multirow{2}{*}{Dimensionless Equation?}  & $\dfrac{\mathrm{d^2}\Psi(\overline{x})}{\mathrm{d}\overline{x}^2}+\left(\varepsilon-\overline{x}^2\right)\Psi(\overline{x})=0,$  \\
                    & $x=x_0\overline{x}, \; x^2_0=\dfrac{\hbar}{m\Omega}, \; E = \dfrac{\hbar\Omega\varepsilon}{2}.$  \\ 
\midrule
\multirow{2}{*}{Change of Variable?} & $\tau(\overline{x})=\overline{x}, \; \tau'(\overline{x})=1.$  \\ 
                                    & $\xi(s)=s, \; \rho(s)=1, \; \rho'(s)=0.$  \\
\midrule
\multirow{5}{*}{GHE?} & $\phi(s)=1,$ \\ 
                     & $\widetilde{\psi}(s)=0,$  \\ 
                     & $\widetilde{\phi}(s;\varepsilon)=\varepsilon-s^2,$  \\ 
                     & $\widetilde{\omega}(s;\varepsilon)=1,$  \\
                     & $I=\mathbb{R}.$  \\ 
\midrule
\multirow{5}{*}{HDE?} & $\pi(s;\varepsilon)=-s,$  \\ 
                     & $\chi(s;\varepsilon)=e^{-s^2/2},$ \\ 
                     & $\psi(s;\varepsilon)=-2s,$ \\ 
                     & $\lambda(\varepsilon)=\varepsilon-1,$  \\ 
                     & $\omega(s;\varepsilon)=e^{-s^2}.$  \\ 
\midrule
Theorem \ref{eigenTheo}? & Applicable.  \\ 
\midrule
Corollary \ref{coroInt}? & Applicable.  \\ 
\midrule
\multirow{4}{*}{Bound States?} & $\Psi_n(s) = \mathcal{N}_n e^{-s^2/2}H_n(s),$ \\ 
                    & $E_n=\hbar \Omega^2\left(n+1/2\right),$  \\
                    & $n=0,1,\dots,$ \\
                    & $\mathcal{N}^2_n=\dfrac{1}{2^n n!\sqrt{\pi}}x^{-1}_0.$  \\
\bottomrule
\end{tabular}
\caption{NU method for the Quantum Harmonic Oscillator.}
\label{table:oscillator}
\endgroup
\end{table}

\noindent \textbf{Application of the method: Table \ref{table:oscillator}.} Eq. \eqref{harmonicoscillator} corresponds to a GHE where
\begin{equation*}
\phi(s)=1, \quad
\widetilde{\psi}(s)=0, \quad
\widetilde{\phi}(s;\varepsilon)=\varepsilon-s^2,
\end{equation*}
and $\varepsilon$ is the new eigenenergy. 

Firstly, apply the NU reduction outlined in Section \ref{reduction}: for a convenient $k_0(\varepsilon) \in \mathbb{C}$, adequate $\pi(s;\varepsilon)$ must solve
$$
0 = \pi^2(s;\varepsilon) - k_0(\varepsilon)+\varepsilon-s^2.
$$
Clearly, for $\pi(s;\varepsilon)$ to be a polynomial, $k_0(\varepsilon)=\varepsilon$ and, therefore, obtain two possibilities
\begin{enumerate}[]
    \item [(A)] $\pi_1(s;\varepsilon) = +s,$
    \item [(B)] $\pi_2(s;\varepsilon) = -s.$
\end{enumerate}
Following Eqs. \eqref{phi}, \eqref{tau,sigmabar} and \eqref{lambda}, write
\begin{enumerate}[]
    \item [(A)] $\chi_1^2(s;\varepsilon) =\exp{\left(+ s^2\right)}=\omega_1(s), \quad \psi_1(s;\varepsilon)=+2s, \quad \lambda_1(\varepsilon)= \varepsilon + 1,$
    \item [(B)] $\chi_2^2(s;\varepsilon)=\exp{\left(-s^2\right)}=\omega_2(s), \quad \psi_2(s;\varepsilon)=-2s, \quad \lambda_2(\varepsilon)= \varepsilon - 1.$
\end{enumerate}
By Lemma \ref{easesChoiceLemma}, it must hold $$\dfrac{\mathrm{d}\psi}{\mathrm{d}s}(s)<0.$$ Therefore, choose the HDE associated to case (B) which corresponds to Hermite's HDE (see Table \ref{table1}). Hence, NU reduction is successful.

Secondly, we inquire whether Theorem \ref{eigenTheo} is applicable: associated $\omega_2(s)$ clearly satisfies Eq. \eqref{orthoPol} for Classical OPs. Hence, Corollary \ref{coroInt} is applicable if the estimate in Eq. \eqref{estimate} holds.

Finally, $\tau(\overline{x}) \equiv \overline{x}$ and, by Lemma \ref{lemmaEasesInt}, $\widetilde{\omega}(s;\varepsilon) \equiv 1$. So, Eq. \eqref{estimate} holds and, by Corollary \ref{coroInt}, the only bound states $(\Psi_n(x),E_n)$ of Problem \ref{schroProb} associated to Eq. \eqref{harmonicoscillatorreal} read
\begin{gather}
    \Psi_n(x) = \mathcal{N}_n e^{-s(x)^2/2}H_n(s(x)), \quad s(x)=xx^{-1}_0, \quad x_0=\sqrt{\dfrac{\hbar}{m\Omega}}, \label{eigenfuncHO}\\
    E_n=\hbar \Omega\left(n+\dfrac{1}{2}\right), \label{eigenenerHO}
\end{gather}
for $n=0,1,\dots$ In Eq. \eqref{eigenfuncHO}, $(H_n(s))_n$ correspond to the Hermite polynomials that can be obtained by its Rodrigues formula $$H_n(s)=B_ne^{s^2}\dfrac{\mathrm{d}^n}{\mathrm{d}s^n}\left(e^{-s^2}\right).$$ In addition, $\mathcal{N}_n$ is a normalization constant which satisfies
\begin{equation*}
\int_{\mathbb{R}}\Psi_n^2(x)\mathrm{d}x = 1.
\end{equation*}
It can be obtained by using data presented in the last row of Table 1. In particular,
$$
\mathcal{N}^2_n=\dfrac{1}{2^n n!\sqrt{\pi}}x^{-1}_0.
$$
As for Eq. \eqref{eigenenerHO}, the eigenvalues are determined by Eq. \eqref{eigenvareps} which particularizes to (see Table \ref{table1} for Hermite polynomials) $$\varepsilon_n-1 = 2n = \lambda_n, \quad E_n = \dfrac{\hbar\Omega\varepsilon_n}{2}.$$
Thus, we achieve a discretization of the permitted energy levels, which represents one of the primary distinctions between classical and quantum theory. This outcome arises from the requirement that 
$\lambda \in \mathbb{C}$ must be such that the HDE admits polynomial solutions, corresponding to Hermite polynomials for this system.

Data for the NU method applied to the Quantum Harmonic Oscillator can be found in Table \ref{table:oscillator}. Finally, it is a standard result that such $\left(\Psi_n\right)^{\infty}_n$ form a COS for $L^{2}(\mathbb{R})$, thus, the Quantum Harmonic Oscillator is observable.

\subsection{Morse Potential} \label{morsePot} It corresponds to \cite{morse}
\begin{equation}
\label{MP}
    V(x) = D_e\left(1-e^{-a\left(x-x_e\right)}\right)^2, \quad x \in \mathbb{R},
\end{equation}
where $x$ is the distance between atoms, $x_e$ is the equilibrium bond distance, $D_e$ is the well depth and $a$ controls the width of the potential. It is a model for the potential energy of diatomic molecules, offering a more accurate representation of vibrational structures than the quantum harmonic oscillator by accounting for bond breaking, anharmonicity, and overtone transitions.

\begin{table}[htpb!]
\centering
\begingroup
\renewcommand{\arraystretch}{1}
\begin{tabular}{r c}
\multirow{1}{*}{Morse Potential?}  & $V(x) = D_e\left(1-e^{-a\left(x-x_e\right)}\right)^2.$  \\ 
\midrule
\multirow{3}{*}{Potential, $V(x)$?}  & $V_{\text{min}}=0,$ \\ 
                    & $V_{-}=\infty,\,V_{+}=D_e,$  \\ 
                    & $V_{\text{max}}=\infty.$  \\ 
\midrule
\multirow{2}{*}{Energy Regions?} & Bound states: $E \in (0,D_e)$,  \\ 
                                & Scattering states: $E \in (D_e,\infty)$.  \\
\midrule
\multirow{2}{*}{Dimensionless Equation?}  & $v(\overline{x})=\Lambda^2\left(1-be^{-\overline{x}}\right)^2, \; \overline{x}=ax,$\\
                    & $b=e^{ax_e}, \; \Lambda^2=\dfrac{2mD_e}{a^2\hbar^2}, \; \varepsilon=\dfrac{\Lambda^2}{D_e}E.$  \\ 
\midrule
\multirow{2}{*}{Change of Variable?} & $\tau(\overline{x})=2 \Lambda be^{-\overline{x}}, \; \tau'(\overline{x})=-2 \Lambda be^{-\overline{x}}.$  \\ 
                                    & $\rho(s)=-\dfrac{1}{s}, \; \rho'(s)=\dfrac{1}{s^2}.$  \\
\midrule
\multirow{5}{*}{GHE?} & $\phi(s)=s,$ \\ 
                     & $\widetilde{\psi}(s)=1,$  \\ 
                     & $\widetilde{\phi}(s;\varepsilon)=\varepsilon-\Lambda^2\left(1-\dfrac{s}{2\Lambda}\right)^2,$  \\ 
                     & $\widetilde{\omega}(s;\varepsilon)=1,$  \\
                     & $I=(0,\infty).$  \\ 
\midrule
\multirow{5}{*}{HDE?} & $2\pi(s;\varepsilon)=-s+2\sqrt{\Lambda^2-\varepsilon},$  \\ 
                     & $\chi(s;\varepsilon)=e^{-s/2}s^{\sqrt{\Lambda^2-\varepsilon}},$ \\ 
                     & $\psi(s;\varepsilon)=1+2\sqrt{\Lambda^2-\varepsilon}-s,$ \\ 
                     & $2\lambda(\varepsilon)=2\Lambda-2\sqrt{\Lambda^2-\varepsilon}-1,$  \\ 
                     & $\omega(s;\varepsilon)=e^{-s}s^{2\sqrt{\Lambda^2-\varepsilon}}.$  \\ 
\midrule
Theorem \ref{eigenTheo}? & Applicable.  \\ 
\midrule
Corollary \ref{coroInt}? & Not Applicable.  \\ 
\midrule
\multirow{4}{*}{Bound States?} & $\Psi_n(s) = \mathcal{N}_n e^{-s/2}s^{\Lambda-n-\frac{1}{2}}L^{(2\Lambda-2n-1)}_n(s),$ \\ 
                    & $\varepsilon_n=\Lambda^2-\left(1/2+n-\Lambda\right)^2,$  \\
                    & $n=0,1,\dots,\lfloor \Lambda-1/2\rfloor$, \\
                    & $\mathcal{N}^2_n=\dfrac{n!(2\Lambda-2n-1)a}{\Gamma(2\Lambda-n)}.$  \\
\midrule
Scattering states? & Empty.  \\ 
\bottomrule
\end{tabular}
\caption{NU method for the Morse Potential.}
\label{table:morse}
\endgroup
\end{table}

We may assume it onedimensional as vibrational dynamics can be detached from the rest. In any case, introduce the following new variables
\begin{equation*}
    \overline{x}=ax, \quad \overline{x}_e=ax_e, \quad \Lambda^2=\dfrac{2mD_e}{a^2\hbar^2}, \quad \varepsilon=\dfrac{\Lambda^2}{D_e}E, \quad b=e^{\overline{x}_e},
\end{equation*}
to obtain Eq. \eqref{schroadi} for
\begin{equation*}
    v(\overline{x})=\Lambda^2\left(1-be^{-\overline{x}}\right)^2, \quad \overline{x} \in \mathbb{R}.
\end{equation*}
It satisfies $v_{-}=v_{\text{max}}=\infty,\, v_{\text{min}}=0,$ and $v_{+}=\Lambda^2.$ Thus, there are two energy region according to Eigenvalue Problem \ref{schroProb}: bound states for $0<\varepsilon <\Lambda^2$ and, scattering states for $\varepsilon > \Lambda^2$. 

\noindent \textbf{Application of the method: Table \ref{table:morse}.} We need to find the {\it universal} change of variable. It is straightforward to see that $v(\overline{x})$ is a polynomial of degree 2 in $e^{-\overline{x}}$. Hence, set $\tau(\overline{x})=2 \Lambda b e^{-\overline{x}}$: constant is chosen so that final HDE corresponds to a canonical Laguerre's HDE (see Table \ref{table1}). Obtain the following GHE:
\begin{equation*}
    \dfrac{\mathrm{d}^2\Psi}{\mathrm{d}s^2}(s)+\dfrac{1}{s}\dfrac{\mathrm{d}\psi}{\mathrm{d}s}(s)+\dfrac{4\varepsilon-4\Lambda^2 +4\Lambda s - s^2}{4s^2}\Psi(s)=0, \quad s \in (0,\infty).
\end{equation*}

Apply NU reduction (see Section \ref{reduction}): $P_2(s;k)$ reads $$P_2(s;k) = \dfrac{1}{4}\left(s^2-4(\Lambda-k)s+4(\Lambda^2-\varepsilon)\right).$$ For it to be a perfect square, $k_0(\varepsilon)=\Lambda \mp \sqrt{\Lambda^2-\varepsilon}$ and, therefore, obtain four possibilities
\begin{enumerate}[]
    \item [(A)] $2\pi_1(s;\varepsilon)=+\left(s+2\sqrt{\Lambda^2-\varepsilon}\right), \quad k^{+}_0(\varepsilon)=\Lambda + \sqrt{\Lambda^2-\varepsilon},$
    \item [(B)] $2\pi_2(s;\varepsilon)=-\left(s+2\sqrt{\Lambda^2-\varepsilon}\right), \quad k^{+}_0(\varepsilon)=\Lambda + \sqrt{\Lambda^2+\varepsilon},$
    \item [(C)] $2\pi_3(s;\varepsilon)=+\left(s-2\sqrt{\Lambda^2-\varepsilon}\right), \quad k^{-}_0(\varepsilon)=\Lambda - \sqrt{\Lambda^2-\varepsilon},$
    \item [(D)] $2\pi_4(s;\varepsilon)=-\left(s-2\sqrt{\Lambda^2-\varepsilon}\right), \quad k^{-}_0(\varepsilon)=\Lambda - \sqrt{\Lambda^2-\varepsilon}.$
\end{enumerate}
Following Eqs. \eqref{phi} and \eqref{tau,sigmabar}, write
\begin{enumerate}[]
    \item [(A)] $\chi_1^2(s;\varepsilon) =e^{+s}s^{+2\sqrt{\Lambda^2-\varepsilon}}=\omega_1(s), \quad \psi_1(s;\varepsilon)=1+2\sqrt{\Lambda^2-\varepsilon}+s,$
    \item [(B)] $\chi_2^2(s;\varepsilon) =e^{-s}s^{-2\sqrt{\Lambda^2-\varepsilon}}=\omega_2(s), \quad \psi_2(s;\varepsilon)=1-2\sqrt{\Lambda^2-\varepsilon}-s,$
    \item [(C)] $\chi_3^2(s;\varepsilon) =e^{+s}s^{-2\sqrt{\Lambda^2-\varepsilon}}=\omega_3(s), \quad \psi_3(s;\varepsilon)=1-2\sqrt{\Lambda^2-\varepsilon}+s,$
    \item [(D)] $\chi_4^2(s;\varepsilon) =e^{-s}s^{+2\sqrt{\Lambda^2-\varepsilon}}=\omega_4(s), \quad \psi_4(s;\varepsilon)=1+2\sqrt{\Lambda^2-\varepsilon}-s.$
\end{enumerate}
Compare with Laguerre's HDE in Table \ref{table1} and choose HDE associated to case (D), $\alpha(\varepsilon) = 2\sqrt{\Lambda^2-\varepsilon}>0$ whenever $\varepsilon < \Lambda^2$, i.e, for the bound states region (remember shape of regions \ref{R1} and \ref{R2}). Hence, NU reduction is successful and Theorem \ref{eigenTheo} is applicable.

Corollary \ref{coroInt} is not applicable as $|\tau'(\overline{x})|$ is not bounded in $\mathbb{R}$ and $\widetilde{\omega}(\overline{x},\varepsilon) \equiv 1$. In any case, apply Theorem \ref{eigenTheo} to obtain pairs $(\Psi_n(x),E_n)$ defined by
\begin{gather}
    \Psi_n(x) = \mathcal{N}_n e^{-s(x)/2}s(x)^{\Lambda-n-\frac{1}{2}}L^{(2\Lambda-2n-1)}_n(s(x)), \quad s(x)=2 \Lambda be^{-ax},\label{eigenfuncMP}\\
    E_n=\dfrac{D_e}{\Lambda^2}\left(\Lambda^2-\left(1/2+n-\Lambda\right)^2\right), \label{eigenenerMP}
\end{gather}
for\footnote{Here, $\lfloor x \rfloor$ denotes largest integer smaller than $x$.} $n=0,1,\dots,\lfloor \Lambda-1/2\rfloor$. Again, in Eq. \eqref{eigenfuncMP}, $L_n^{(\alpha)}(s)$ correspond to the Laguerre polynomial that can be generated by Eq. \eqref{rogriguesFormula} plus the data available in Table \ref{table1}. As for the eigenenergies in Eq. \eqref{eigenenerMP}, they result from imposing associated Eq. \eqref{eigenEq} which translates to $\lambda(\varepsilon_n)=n$ for $n=0,1,\dots$, see Table \ref{table:morse}. Notice that they are finite: this condition appears by explicitly imposing $0< \varepsilon < \Lambda$ in the corresponding eigenenergy equation given in Lemma \ref{eigenvareps}. Nonetheless, at this point, they are but candidates to bound states: they solve the SE associated to Morse Potential given by Eq. \eqref{MP} and nothing more. 

However, using last row in Table \ref{table1}, compute
\begin{equation}
\label{intCondMP}
    \begin{split}
       &\int_{\mathbb{R}}\Psi^2_n(x)\mathrm{d}x = \int^{\infty}_{0}\Psi^2_n(s)\dfrac{\mathrm{d}s}{s} = \\
       & \quad \quad \quad \int^{\infty}_{0}e^{-s}s^{2\Lambda-2n-2}\left(L^{(2\Lambda-2n-1)}_n(s)\right)^2\mathrm{d}s=\dfrac{\Gamma(2\Lambda-n)}{n!(2\Lambda-2n-1)a}, 
    \end{split}
\end{equation}
which is finite by condition $n<\Lambda-1/2$ and, so, such solutions are actual bound states. In addition, there can be no more bound states. Indeed, comparing HDE given by previous Case (D) with CHE given by Eq. \eqref{confluentHyperEq}, find parameters
\begin{equation}
\label{parMP}
    2a(\varepsilon) = 2\sqrt{\Lambda^2-\varepsilon}+1-2\Lambda=-2\lambda(\varepsilon), \quad c(\varepsilon) = 1+2\sqrt{\Lambda^2-\varepsilon}>0.
\end{equation}
Following Subsection \ref{confluent}, find the two linearly independent solutions of such CHE. Define\footnote{This quantity is related to the asymptotic linear momentum at $x \to \infty$. Hence, the subscript.} $\varkappa_+(\varepsilon) \coloneqq \sqrt{\Lambda^2-\varepsilon}$. 

Firstly, assume $0 < 2\varkappa_+(\varepsilon) \notin \mathbb{Z}$. Then, $c(\varepsilon) \notin \mathbb{Z}$ and, by Proposition \ref{conlfuentProp} or Eq. \eqref{wro1c}, the two linearly independent solutions of the original SE, for each $0 < \varepsilon < \Lambda^2$, read
\begin{equation}
\label{wavesMP}
\begin{split}
&\Psi_{1}(s;\varepsilon) = e^{-s/2}s^{\varkappa_+(\varepsilon)}\hypergeometricsetup{symbol=\mathbf{F}}\pFq[skip=4]{1}{1}{a(\varepsilon)}{c(\varepsilon)}{s}, \\
&\Psi_{2}(s;\varepsilon) = e^{-s/2}s^{-\varkappa_+(\varepsilon)}\hypergeometricsetup{symbol=\mathbf{F}}\pFq[skip=4]{1}{1}{a(\varepsilon)-2\varkappa_+(\varepsilon)}{2-c(\varepsilon)}{s}.
\end{split}
\end{equation}
Assume further that $a(\varepsilon)$, $a(\varepsilon)-2\varkappa_+(\varepsilon) \neq 0,-1,\dots$. Then, use Eq. \eqref{inf1} and find
\begin{equation}
\label{asymptoticEigenMP}
    \Psi_1(s;\varepsilon) \sim \dfrac{1}{\Gamma(a(\varepsilon))}\dfrac{e^{s/2}}{s^{1/2+\Lambda}},\quad \Psi_2(s;\varepsilon) \sim \dfrac{1}{\Gamma(a(\varepsilon)-2\varkappa_+(\varepsilon))}\dfrac{e^{s/2}}{s^{1/2+\Lambda}},
\end{equation}
whenever $s\to\infty$. Hence, they are unbounded and not of integrable square for these particular $\varepsilon$. Now, assume, $a(\varepsilon_n) = -n=-\lambda(\varepsilon)$ for some $n=0,1,\dots$ Then, by initial assumption $2\varkappa_+(\varepsilon) \notin \mathbb{Z}$, $a(\varepsilon_n)-2\varkappa_+(\varepsilon_n) \notin \mathbb{Z}^{-}$: following Eq. \eqref{specialCase1Confluent}, $\Psi_1(s;\varepsilon_n)$ reduces to $\Psi_n(s)$ given by Eq. \eqref{eigenfuncMP} and $\Psi_2(s;\varepsilon_n)$ is unbounded by Eq. \eqref{asymptoticEigenMP} again. Finally, assume $a(\varepsilon_m)-2\varkappa_+(\varepsilon_m) = -m$ for some $m=0,1,\dots$ Similarly, $a(\varepsilon_m) \notin \mathbb{Z}^{-}$: $\Psi_1(s;\varepsilon_m)$ is unbounded by Eq. \eqref{asymptoticEigenMP} and $\Psi_2(s;\varepsilon_m)$ reduces, again, to $\Psi_m(s)$ given by Eq. \eqref{eigenfuncMP}. Both $a(\varepsilon)$ and $a(\varepsilon)-2\varkappa_+(\varepsilon)$ cannot be an integer simultaneously by assumption. Hence, case $2\varkappa_+(\varepsilon) \notin \mathbb{Z}$ is covered and no further bound states are obtained.

Secondly, assume $k=2\varkappa_+(\varepsilon_k) =1,2,\dots$ Hence, Proposition \ref{conlfuentProp} is no longer valid. However, $1+k=c(\varepsilon_k)>1$ and, so, $\Psi_1(s;\varepsilon_k)$ is still a solution. As for $\Psi_2(s;\varepsilon_k)$, it needs to be rewritten. Assume $a(\varepsilon_k) \neq 0,-1,\dots$, so, $\Psi_1(s;\varepsilon)$ is, again, not of integrable square by Eq. \eqref{asymptoticEigenMP}. Following Eq. \eqref{wro2c}, a second linearly independent solution reads
\begin{equation}
\label{waveMP2}
\Psi_{2}(s;\varepsilon_k) = e^{-s/2}s^{k/2}\hypergeometricsetup{symbol=G}\pFq[skip=4]{1}{1}{a(\varepsilon_k)}{1+k}{s},\quad a(\varepsilon_k)=\dfrac{k}{2}+\dfrac{1}{2}-\Lambda.
\end{equation}
Now, $c(\varepsilon_k) \ge 2$ and, by Eqs. \eqref{lzc2}, \eqref{lzc3} and \eqref{lzc4}: whenever $a(\varepsilon_k)=k-k'$ for $k'=0,1,\dots,k-1,$
\begin{align*}
    e^{s/2}\Psi_{2}(s;\varepsilon_k)&=O\!\left(s^{1-k/2}\right),
\end{align*}
whenever $s \to 0^{+}$. In addition, $a(\varepsilon_1)=1-\Lambda=1$ and, so, $\Lambda=0$ which is an absurd; case $k=1$ is excluded in this scenario. For the rest of them,
\begin{align*}
    e^{s/2}\Psi_{2}(s;\varepsilon_k) &= \dfrac{\Gamma(k)}{\Gamma\left(\frac{k+1}{2}-\Lambda\right)}s^{-k}+O\!\left(s^{1-k}\right), \quad k=2,3,\dots,\\
    e^{s/2}\Psi_{2}(s;\varepsilon_k) &= \dfrac{1}{\Gamma(1-\Lambda)}s^{-1/2}+O\!\left(s^{-1/2}\log{s}\right), \quad k=1,
\end{align*}
whenever $s \to 0^+$. It is straightforward to corroborate that, in every situation, $\Psi^2_{2}(s;\varepsilon_k) \sim e^{-s}s^{\alpha+1}$ as $s \to 0^{+}$ for some $\alpha \le -1$. Following Eq. \eqref{intCondMP}, $\Psi_2(s;\varepsilon_k)$ is not square-integrable. Hence, there are no more bound states whenever $a(\varepsilon_k) \neq 0,-1,\dots$ Finally, assume $a(\varepsilon_k) = -k'$ for $k'=0,1,\dots$ Then, again, $\Psi_1(s;\varepsilon_k)$ in Eq. \eqref{wavesMP} reduces to $\Psi_{k'}(s)$ given in Eq. \eqref{eigenfuncMP}. As for $\Psi_2(s;\varepsilon_k)$, it needs to be rewritten, again. Combining Eqs. \eqref{wro1c} and \eqref{wro5c},
\begin{equation}
\label{waveMP3}
\Psi_{2}(s;\varepsilon_k) = e^{s/2}s^{k/2}\hypergeometricsetup{symbol=G}\pFq[skip=4]{1}{1}{1+k-k'}{1+k}{e^{\pi i}s},
\end{equation}
is another linearly independent solution which is clearly not square-integrable by Eq. \eqref{inf2}. Hence, case $2\varkappa_+(\varepsilon) \in \mathbb{Z}$ is covered and no further bound states are obtained. The only bound states for the Morse Potential are given by Eqs. \eqref{eigenfuncMP} and \eqref{eigenenerMP}. Notice that this discussion arises from the impossibility of applying Corollary \ref{coroInt}.

Now, proceed further to the scattering states region, $\varepsilon > \Lambda^2$. In this region, both $a(\varepsilon)$ and $c(\varepsilon)$ in Eq. \eqref{parMP} have a non vanishing imaginary part. Hence, $\Psi_1(s;\varepsilon)$ and $\Psi_2(s;\varepsilon)$ given by Eq. \eqref{wavesMP} exist and are linearly independent. However, Eq. \eqref{asymptoticEigenMP} is still valid and, so, they are unbounded. Therefore, there are no possible scattering states and, hence, {\it observability} cannot be obtained: the space of states cannot constitute a COS for $L^{2}(\mathbb{R})$. This is mildly {\it traumatising} as it is a property needed to ``practice'' physics on the system. 

Although the potential is treated as one-dimensional following \cite[Table 4.1, pp. 296–297]{susy}, this implies that it cannot be fully understood as a physical system on its own. This suggests that we should instead consider the system as three-dimensional, which is commonly done. However, this one-dimensional approach does provide the correct bound energy levels. In any case, it simply constitutes an example of a physical system solved in terms of Laguerre polynomials.

\subsection{Rosen-Morse II Potential} Its dimensionless potential corresponds to \cite{rm}
\begin{equation}
	v(\overline{x}) =  v_{0} \cosh^{2}{\mu}\left(\tanh{\overline{x}} + \tanh{\mu} \right)^{2}, \quad \overline{x}\in \mathbb{R}.
	\label{rmpot}
\end{equation}
It satisfies $v_{\pm}=v_{0}e^{\pm 2 \mu},\, v_{\text{min}}=0,$ and $v_{\text{max}}=v_+$. There are two energy region according to Eigenvalue Problem \ref{schroProb}: bound states for $0 < \varepsilon < v_{0}e^{-2\mu}$ and, scattering states for $\varepsilon > v_{0}e^{-2\mu}$. In addition, scattering states can be further subdivided into reflecting states: $v_{0}e^{-2\mu} < \varepsilon < v_{0}e^{+2\mu}$; and free states: $\varepsilon > v_{0}e^{2\mu}$. We clarify their nature once the spectrum is solved.

\begin{table}[htpb!]
\centering
\begingroup
\renewcommand{\arraystretch}{1}
\begin{tabular}{r l}
\multirow{1}{*}{Rosen-Morse II Potential?}  & $v(\overline{x}) =  v_{0} \cosh^{2}{\mu}\left(\tanh{\overline{x}} + \tanh{\mu} \right)^{2}.$  \\ 
\midrule
\multirow{3}{*}{Potential, $v(x)$?}  & $v_{\text{min}}=0,$ \\ 
                    & $v_{\pm}=v_{0}e^{\pm 2 \mu},$  \\ 
                    & $v_{\text{max}}=v_+.$  \\ 
\midrule
\multirow{3}{*}{Energy Regions?} & Bound states: $\varepsilon \in (0,v_0e^{-2\mu})$,  \\ 
                                & Reflecting states: $\varepsilon \in (v_0e^{-2\mu},v_0e^{2\mu})$,  \\
                                & Free states: $\varepsilon \in (v_0e^{2\mu},\infty)$.  \\
\midrule
\multirow{2}{*}{Change of Variable?} & $\tau(\overline{x})=-\tanh{\overline{x}}, \; \tau'(\overline{x})=\tanh^2{\overline{x}}-1.$  \\ 
                                    & $\rho(s)=-\dfrac{1}{1-s^2}, \; \rho'(s)=-\dfrac{2s}{(1-s^2)^2}.$  \\
\midrule
\multirow{5}{*}{GHE?} & $\phi(s)=1-s^2,$ \\ 
                     & $\widetilde{\psi}(s)=-2s,$  \\ 
                     & $\widetilde{\phi}(s;\varepsilon)=\varepsilon-v_{0}\cosh^{2}{\mu}\left(s - \tanh{\mu}\right)^{2},$  \\ 
                     & $\widetilde{\omega}(s;\varepsilon)=1,$  \\
                     & $I=(-1,1).$  \\ 
\midrule
\multirow{9}{*}{HDE?} & $\varkappa_+(\varepsilon)=\sqrt{v_{+}-\varepsilon},\; \varkappa_-(\varepsilon)= \sqrt{v_{-}-\varepsilon},$ \\
                    & $2\mathrm{A}(\varepsilon)=\varkappa_+(\varepsilon) - \varkappa_-(\varepsilon),$ \\
                    & $2\mathrm{B}(\varepsilon)=\varkappa_+(\varepsilon) + \varkappa_-(\varepsilon).$ \\
                    &$\pi(s;\varepsilon)=\mathrm{A}(\varepsilon)-\mathrm{B}(\varepsilon)s,$  \\ 
                     & $\chi^2(s;\varepsilon)=\omega(s;\varepsilon),$ \\ 
                     & $\psi(s;\varepsilon)=2\mathrm{A}(\varepsilon)-2\left(\mathrm{B}(\varepsilon)+1\right)s,$ \\ 
                    &$2k(\varepsilon)=\varepsilon+v_0-\varkappa_+(\varepsilon) \varkappa_-(\varepsilon),$  \\ 
                     & $\omega(s;\varepsilon)=\left(1-s\right)^{\varkappa_-(\varepsilon)}\left(1+s\right)^{\varkappa_+(\varepsilon)}.$  \\ 
\midrule
Theorem \ref{eigenTheo}? & Applicable.  \\ 
\midrule
Corollary \ref{coroInt}? & Applicable.  \\ 
\midrule
\multirow{6}{*}{Bound States?}
                    & $v_1 = \frac{1}{2}v_0 \sinh{2 \mu}, \; v_2 = v_0\cosh^2{\mu}+\frac{1}{4},$ \\
                    & $B_n\mathrm{A}_n=v_1, \; B_n=\sqrt{v_2}-\left(n+\frac{1}{2}\right).$ \\
                    & $\begin{aligned}
                    \Psi_n(\overline{x}) &= \mathcal{N}_n e^{-\mathrm{A}_n \overline{x}} \text{sech}^{B_n}(\overline{x})\\
                    &\quad \times P_{n}^{(B_n-\mathrm{A}_n,\,B_n+\mathrm{A}_n)}\left(-\tanh{\overline{x}}\right),
                    \end{aligned}$ \\ 
                    & $\varepsilon_n=v_--(B_n-\mathrm{A}_n)^2,\; n=0,1,\dots,\lfloor N(v_0,\mu) \rfloor,$ \\
                    & $N(v_0,\mu) = \sqrt{v_2}-\sqrt{v_1}-\frac{1}{2},$  \\
                    & For $\mathcal{N}_n$, see \cite[Appendix B]{GORDILLONUNEZ2024134008}. \\
\midrule
\multirow{3}{*}{GHDE?} & $a(\varepsilon) = B(\varepsilon)+\frac{1}{2}-\sqrt{v_2},$ \\ 
&$b(\varepsilon) = B(\varepsilon)+\frac{1}{2}+\sqrt{v_2},$ \\ &$c(\varepsilon) = \varkappa_{+}(\varepsilon)+1.$  \\
\midrule
\multirow{4}{*}{Scattering States?} & For $v_-< \varepsilon < v_+$, $\Psi_{\varepsilon}(\overline{x}) \propto \Psi_1(\overline{x};\varepsilon),$ \\ 
& For $v_+< \varepsilon$, $\Psi_{\varepsilon}(\overline{x}) \propto \Psi_1(\overline{x};\varepsilon)+C\Psi_2(\overline{x};\varepsilon),$ \\ 
&$\psi_{1}(\overline{x};\varepsilon) = \dfrac{\text{sech}^{b(\varepsilon)}{\overline{x}}}{e^{a(\varepsilon) \overline{x}}}  \hypergeometricsetup{symbol=F, fences=brack}\pFq[skip=4]{2}{1}{\displaystyle a(\varepsilon),\mathrm{b}(\varepsilon)}{\displaystyle c(\varepsilon)}{t(\overline{x})},$  \\ 
&$\psi_{2}(\overline{x};\varepsilon) = \dfrac{e^{a(\varepsilon) \overline{x}}}{\text{sech}^{b(\varepsilon)}{\overline{x}}}  \hypergeometricsetup{symbol=F, fences=brack}\pFq[skip=4]{2}{1}{\displaystyle a'(\varepsilon),b'(\varepsilon)}{\displaystyle c'(\varepsilon)}{t(\overline{x})},$  \\
\bottomrule
\end{tabular}
\caption{NU method for the Rosen-Morse II Potential.}
\label{table:rosenmorseII}
\endgroup
\end{table}

\noindent \textbf{Application of the method: Table \ref{table:rosenmorseII}.} Set $\tau(\overline{x})=-\tanh{\overline{x}}$. Obtain its corresponding GHE, see Table \ref{table:rosenmorseII}.

Apply NU reduction, follow Section \ref{reduction}. Polynomial $\pi(s;\varepsilon)$ solves
\begin{equation}
\label{pi-eq-rmii}
    \pi^2(s;\varepsilon)+[\widetilde{\phi}(s;\varepsilon)-k(\varepsilon)\phi(s;\varepsilon)]=0.
\end{equation}
Define 
\begin{equation*}
v_1 \coloneqq \frac{1}{2}v_0 \sinh{2 \mu}, \quad v_2 \coloneqq v_0\cosh^2{\mu}+\frac{1}{4}, \quad \varkappa_{\pm}(\varepsilon) \coloneqq \sqrt{v_{\pm}-\varepsilon}>0.
\end{equation*}
Hence, introducing $\pi(s;\varepsilon)=\mathrm{A}(\varepsilon)-\mathrm{B}(\varepsilon)s$ in Eq. \eqref{pi-eq-rmii}, write
\begin{equation}
\label{sys-abk}
\mathrm{A}(\varepsilon)^2-k(\varepsilon)=v_2+\dfrac{1}{4}-\varepsilon-v_0,\quad \mathrm{A}(\varepsilon)\mathrm{B}(\varepsilon)=v_1,\quad \mathrm{B}(\varepsilon)^2+k(\varepsilon)=v_2-\dfrac{1}{4}.
\end{equation}
Obtain, firstly,
\begin{equation*}
    \begin{aligned}
        2A_1(\varepsilon)=\varkappa_+&(\varepsilon) - \varkappa_-(\varepsilon), \quad 2B_1(\varepsilon)=\varkappa_+(\varepsilon) + \varkappa_-(\varepsilon), \\
        &2k_1(\varepsilon)= \varepsilon+v_0-\varkappa_+(\varepsilon) \varkappa_-(\varepsilon).
    \end{aligned}
\end{equation*}
From it, the remaining three possibilities are obtained as a direct consequence of two symmetries: $(i)$ Eq.
\eqref{pi-eq-rmii}  is invariant under the transformation $\pi(s;\varepsilon) \mapsto -\pi(s;\varepsilon)$ for a fixed $k(\varepsilon)$; and
$(ii)$ the system of Eqs. \eqref{sys-abk} is also invariant if we interchange $\mathrm{A}(\varepsilon)$ and $\mathrm{B}(\varepsilon)$ and simultaneously replace $k(\varepsilon)$ by $-k(\varepsilon)+v_0+\varepsilon$. However, by Lemma \ref{easesChoiceLemma}, we discard them (and drop the subscript), see Table \ref{table:rosenmorseII}: NU reduction is successful. 

Comparing with Jacobi's HDE in Table \ref{table1}, $\alpha(\varepsilon)=\varkappa_{-}(\varepsilon)$ and $\beta(\varepsilon)=\varkappa_{+}(\varepsilon)$. For $\varepsilon < v_-$, $\alpha(\varepsilon),\,\beta(\varepsilon) > 0$ and, therefore, Theorem \ref{eigenTheo} is applicable. Finally, $\tau'(\overline{x})=\tanh^2{\overline{x}}-1$, which is bounded, and so, Corollary \ref{coroInt} is applicable: find bound states in Table \ref{table:rosenmorseII}. To find $\varepsilon_n$, obtain, firstly, $\mathrm{B}_n \coloneqq \mathrm{B}(\varepsilon_n)$ by combining last equality in Eq. \eqref{sys-abk} with the associated eigenenergy equation
$$k_n-\mathrm{B}_n = n(n+2\mathrm{B}_n+1), \quad k_n \coloneqq k(\varepsilon_n).$$

Finally, proceed into the scattering states, $\varepsilon > v_-$. Set $s=-1+2t$ and transform previous Jacobi's HDE into a GHDE, see Eq. \eqref{hyperEq}, where $t \in (0,1)$ and the parameters are
\begin{equation}
\begin{split}
a(\varepsilon) = B(\varepsilon)+\frac{1}{2}-\sqrt{v_2},& \quad b(\varepsilon) = B(\varepsilon)+\frac{1}{2}+\sqrt{v_2}, \\ 
c(\varepsilon) = A(\varepsilon)+B(\varepsilon)&+1=\varkappa_{+}(\varepsilon)+1.
\label{abgGausshdeMF}
\end{split}
\end{equation}
Following Subsection \ref{gauss}, construct the two linearly independent solutions. 

Firstly, set $\varepsilon<v_+$ and, so, $c(\varepsilon) > 0$. Assume $\varkappa_+(\varepsilon) \notin \mathbb{Z}$. Then, $c(\varepsilon) \notin \mathbb{Z}$ and, by Proposition \ref{gaussProp}, the two linearly independent solutions of the original SE read
\begin{equation}   
\begin{split}
\Psi_1(t;\varepsilon) & = \left(1-t\right)^{\frac{i}{2}\left|\varkappa_-(\varepsilon)\right|}t^{+\frac{1}{2}\varkappa_+(\varepsilon)}\pFq[skip=4]{2}{1}{a(\varepsilon),b(\varepsilon)}{c(\varepsilon)}{t}, \\
\Psi_2(t;\varepsilon) & = \left(1-t\right)^{\frac{i}{2}\left|\varkappa_-(\varepsilon)\right|}t^{-\frac{1}{2}\varkappa_+(\varepsilon)}\pFq[skip=4]{2}{1}{a'(\varepsilon),b'(\varepsilon)}{c'(\varepsilon)}{t},
\label{posWaveFunctionR2}
\end{split}
\end{equation}
where, as usual, see Eq. \eqref{primeParameters}, $a'(\varepsilon)=a(\varepsilon)-c(\varepsilon)+1,\,b'(\varepsilon)=b(\varepsilon)-c(\varepsilon)+1$, $c'(\varepsilon)=2-c(\varepsilon)$, and hypergeometric functions are understood as series, see Eq. \eqref{hyperSeries}. In addition, $$c(\varepsilon)-a(\varepsilon)-b(\varepsilon) = c'(\varepsilon)-a'(\varepsilon)-b'(\varepsilon) = A(\varepsilon)-B(\varepsilon)=-i\left|\varkappa_-(\varepsilon)\right|,$$
and, so, by Eq. \eqref{lfz13}, both $\Psi_1(t;\varepsilon)$ and $\Psi_2(t;\varepsilon)$ are bounded at $t=1^-$. However, $\Psi_2(t;\varepsilon)$ is clearly unbounded at $t=0^+$. Now, assume $\varkappa_{+}(\varepsilon_m)=m$ for some $m=1,2,\dots$ Then, $\Psi_2(t;\varepsilon)$ needs to be rewritten. Accordingly to \cite[\S21.4, page 277]{Nikiforov1988SpecialFO}, since, $c(\varepsilon_m)$ is a positive integer,
there are two possibilities:
\begin{itemize}
\item[(A)] $a(\varepsilon_m)$,  $b(\varepsilon_m)$, and $a(\varepsilon_m)+b(\varepsilon_m)$ are not integers, or
\item[(B)] one of them is an integer.
\end{itemize}
However, it is straightforward to check that $\Im(a(\varepsilon_m))\neq0$, $\Im(b(\varepsilon_m))\neq0$, and $\Im(a(\varepsilon_m)+b(\varepsilon_m))\neq0$ for every $m=1,2,\dots$ Hence, the second solution of the associated GHDE is given by, see \cite[\S21.4, page 277]{Nikiforov1988SpecialFO}, $$y_A(t;\varepsilon_m)=\pFq[skip=4]{2}{1}{a(\varepsilon_m),b(\varepsilon_m)}{a(\varepsilon_m)+b(\varepsilon_m)-m}{1-t},$$ 
and, so,
$$
\Psi_A(t;\varepsilon_m)=\left(1-t\right)^{\frac{i}{2}\left|\varkappa_-(\varepsilon_m)\right|}t^{+\frac{1}{2}\varkappa_+(\varepsilon_m)}
 y_A(t;\varepsilon_m)
$$
is a second linearly independent solution of the original SE. However, accordingly to Eq. \eqref{lfz14},
$$
y_A(t;\varepsilon_m) \sim \frac{\Gamma(\alpha+\beta-m)\Gamma(m)}{\Gamma(\alpha)\Gamma(\beta)} t^{-m} \text{ as } t \to 0^{+},
$$
from where it follows that $\psi_A(s;\varepsilon_m)$ is unbounded at $t=0^{+}$. Hence, the only bounded solution for $\varepsilon < v_+$ is $\Psi_1(t;\varepsilon)$ in Eq. \eqref{posWaveFunctionR2}.

Secondly, set $\varepsilon > v_+$. In that case $\Im{\left(c(\varepsilon)\right)} \neq 0$ and, so, by Proposition \ref{gaussProp}, the two linearly independent solutions of the original SE read 
\begin{equation*}   
\begin{split}
\Psi_1(t;\varepsilon) & = \left(1-t\right)^{\frac{i}{2}\left|\varkappa_-(\varepsilon)\right|}t^{+\frac{i}{2}\left|\varkappa_+(\varepsilon)\right|}\pFq[skip=4]{2}{1}{a(\varepsilon),b(\varepsilon)}{c(\varepsilon)}{t}, \\
\Psi_2(t;\varepsilon) & = \left(1-t\right)^{\frac{i}{2}\left|\varkappa_-(\varepsilon)\right|}t^{-\frac{i}{2}\left|\varkappa_+(\varepsilon)\right|}\pFq[skip=4]{2}{1}{a'(\varepsilon),b'(\varepsilon)}{c'(\varepsilon)}{t}.
\end{split}
\end{equation*}
Following similar reasoning, it is straightforward to observe that, in this case, both solutions are bounded. Therefore, for $\varepsilon > v_+$, the eigenvalues are degenerate. See Table \ref{table:rosenmorseII} for the explicit formulas of these functions in terms of the variable $x$.

Notice that, in terms of degeneracy, the two subregions into which we subdivided the scattering region \(\varepsilon > v_-\) behave differently. Whenever \(\varepsilon < v_+\), the hypothetical particle subjected to the potential cannot (classically) reach \(+\infty\). In terms of interpretation, once we leave the bound states region, the solution can no longer be understood as a classical particle but rather as a beam of particles that starts at \( -\infty \) or \( +\infty \) and interacts with the potential through reflection and transmission.

By analyzing the asymptotics of the solutions \(\Psi_1(x;\varepsilon)\) and \(\Psi_2(x;\varepsilon)\), one can see that \(\Psi_1(x;\varepsilon)\) corresponds to a beam of particles starting at \( -\infty \), while \(\Psi_2(x;\varepsilon)\) contains parts of the solution corresponding to beams of particles starting at both \( +\infty \) and \( -\infty \). From this, we can construct another solution, \(\Psi_3(x;\varepsilon)\), which corresponds to a beam of particles starting at \( +\infty \) and propagating to \( -\infty \).

In this context, we observe the difference in degeneracy: whenever \(\varepsilon < v_+\), a beam of particles cannot propagate from \( +\infty \) (its energy is lower than the potential) to \( -\infty \). Hence, the degeneracy.

These aspects, along with the derivation of the orthogonality and completeness relations that justify the observability of the system, and the fact that the space of states constitutes a complete orthonormal system (COS) for \(L^2(\mathbb{R})\), are addressed in \cite[Section 2.5, pp. 17-20]{GORDILLONUNEZ2024134008}.

\section{Conclusions}
\label{conclusions}
In many papers, the Nikiforov-Uvarov method is mainly seen as the reduction from the GHE to the HDE, as we described in Section \ref{reduction}. Here, we aimed to go beyond that and review not only this reduction, but also the specific results that lead to solutions in terms of (old) Classical OPs and hypergeometric functions. We also looked at the results that show when the solutions are square-integrable: Theorem \ref{eigenTheo} plays an important role in this process, but in some cases, like the Morse potential in Subsection \ref{morsePot}, Corollary \ref{coroInt} does not apply. Despite that, Classical OPs generally describe the discrete bound states, while the scattering states are continuous and represented by hypergeometric functions. This could be because bound states in these simple systems are best described by the simplest mathematical objects, like orthogonal polynomials. It is also important to note the discussion in Subsection \ref{morsePot}, especially after Corollary \ref{coroInt} could not be used. The discussion in Subsection \ref{boundstates} helps explain why. Another point is that the reduction itself highlights some physically important quantities. For example, in both the Morse and Rosen-Morse potentials, many parameters are related to $\varkappa_{\pm}(\varepsilon)$, which corresponds to the asymptotic linear momenta at $\pm \infty$.

This method provides a clear and systematic approach to solving these problems, which could also be unified with computer algebra systems. For more on this, see, again, \cite{lineellis}, where they explore many classical systems solved with this method. There are also other quantum systems (both relativistic and non-relativistic) solved using this approach, as shown in \cite{ap2}, \cite{ap5}, \cite{ap4}, \cite{ap6}, \cite{ap3}, \cite{ap1} and references therein. While Corollary \ref{coroInt} may not always apply, we believe the results are still correct, even if the reasoning is slightly different. Finally, the Nikiforov-Uvarov method can also be used to solve problems where square integrability is not required. For examples of such applications, see \cite{aps2} or \cite{aps1}.

\appendix
\renewcommand{\theequation}{A\arabic{equation}}
\setcounter{equation}{0}

\section{Some auxiliary results}
For analytic continuation of the solutions of an equation, we rely on the following theorem on the analyticity of the integral depending on a parameter (see \cite[Theorem 2, \S3, p. 13]{Nikiforov1988SpecialFO}):
\begin{theorem}
\label{theoA.1}
Assume that $C$ is a piecewise smooth curve of finite length and $f: \mathcal{O} \coloneqq C \times \Omega \subset \mathbb{C} \times \mathbb{C} \to \mathbb{C}$ satisfying that $f(s,z) \in C(\mathcal{O})$, $f(s_0,z) \in \mathcal{H}(\Omega)$, for every $s_0 \in C$, then, $F(z)$ given by
$$
F(z) = \int_{C}f(s,z)\mathrm{d}s,
$$
satisfies that $F(z) \in \mathcal{H}(\Omega)$, and
$$
F^{(n)}(z) = \int_{C} \frac{\partial^{n}f}{\partial z^{n}}(s,z)\mathrm{d}s.
$$
The conclusion remains valid for uniformly convergent improper integrals $F(z)$.
\end{theorem}
Finally, we include some basic properties of the hypergeometric functions derived in Subsection \ref{hyperFuncs}. These results are detailed in \cite{NIST} and presented for the sake of completeness.
\subsection{Basic Properties of Hypergeometric Functions\label{basicProp}}
\subsubsection{Hypergeometric Function}We present some basic properties for the hypergeometric function ${}_2F_1(a,b,c;z)$ introduced in Subsection \ref{gauss}:

\begin{custom}{Limiting Form as $z \to 1^{-}$ (\cite[\S 15.4(ii)]{NIST})}
If $\Re{\left(c-a-b\right)>0}$,
\begin{equation*}
\label{lfz11}
    \pFq[skip=4]{2}{1}{a,b}{c}{1}=\dfrac{\Gamma(c)\Gamma(c-a-b)}{\Gamma(c-a)\Gamma(c-b)}.
\end{equation*}
If $c=a+b$,
\begin{equation*}
\label{lfz12}
    \lim_{z \to 1^-}\dfrac{1}{-\log{\left(1-z\right)}}\pFq[skip=4]{2}{1}{a,b}{a+b}{z}=\dfrac{\Gamma(a+b)}{\Gamma(a)\Gamma(b)}.
\end{equation*}
If $\Re(c-a-b)=0$ but $c \neq a+b$,
\begin{equation}
\label{lfz13}
    \lim_{z \to 1^-}\left(1-z\right)^{a+b-c}\left(\pFq[skip=4]{2}{1}{a,b}{c}{z}-\dfrac{\Gamma(c)\Gamma(c-a-b)}{\Gamma(c-a)\Gamma(c-b)}\right)=\dfrac{\Gamma(c)\Gamma(a+b-c)}{\Gamma(a)\Gamma(b)}.
\end{equation}
Finally, if $\Re(c-a-b)<0$,
\begin{equation}
\label{lfz14}
    \lim_{z \to 1^-} \dfrac{1}{(1-z)^{c-a-b}}\pFq[skip=4]{2}{1}{a,b}{c}{z} = \dfrac{\Gamma(c)\Gamma(a+b-c)}{\Gamma(a)\Gamma(b)}.
\end{equation}
\end{custom}
\subsubsection{Confluent Hypergeometric Function}
We present some basic properties for the confluent hypergeometric functions ${}_1F_1(a,c;z)$, ${}_1\mathbf{F}_1(a,c;z)$ and ${}_1G_1(a,c;z)$ introduced in Subsection \ref{confluent}:

\begin{custom}{Limiting Form at $\mathbf{0}$ (\cite[\S 13.2(iii)]{NIST})}
\begin{equation*}
    \hypergeometricsetup{symbol=F}\pFq[skip=4]{1}{1}{a}{c}{z}=1+O(z). \label{lzc0}
\end{equation*}
Whenever $a = -n$ or $-n+c+1$ for a nonnegative integer $n$,
\begin{align}
    \hypergeometricsetup{symbol=G}\pFq[skip=4]{1}{1}{-n}{c}{z}&=(-1)^n(c)_n+O(z), \nonumber\\
    \hypergeometricsetup{symbol=G}\pFq[skip=4]{1}{1}{-n+c-1}{c}{z}&=(-1)^n\left(2-c\right)_nz^{1-c}+O\!\left(z^{2-b}\right). \label{lzc2}
\end{align}
For the rest of cases,
\begin{align}
    \hypergeometricsetup{symbol=G}\pFq[skip=4]{1}{1}{a}{c}{z} &= \dfrac{\Gamma(c-1)}{\Gamma(a)}z^{1-c}+O\!\left(z^{2-\Re\left(c\right)}\right), \quad \Re(c) \ge 2,\; c\neq2, \label{lzc3}\\
    \hypergeometricsetup{symbol=G}\pFq[skip=4]{1}{1}{a}{2}{z} &= \dfrac{1}{\Gamma(a)}z^{-1}+O\!\left(\log{z}\right), \label{lzc4} \\
    \hypergeometricsetup{symbol=G}\pFq[skip=4]{1}{1}{a}{c}{z} &= \dfrac{\Gamma(c-1)}{\Gamma(a)}z^{1-c}+\dfrac{\Gamma(1-c)}{\Gamma(a-c+1)}+O\!\left(z^{2-\Re\left(c\right)}\right), \quad 1 < \Re(c) < 2, \nonumber\\
    \hypergeometricsetup{symbol=G}\pFq[skip=4]{1}{1}{a}{1}{z} &= -\dfrac{1}{\Gamma(a)}\left(\log{z}+\psi(a)+2\gamma\right)+O\!\left(z\log{z}\right), \nonumber\\
    \hypergeometricsetup{symbol=G}\pFq[skip=4]{1}{1}{a}{c}{z} &= \dfrac{\Gamma(1-c)}{\Gamma(a)}+O\!\left(z^{1-\Re\left(c\right)}\right), \quad 0 < \Re(c) < 1, \nonumber\\
    \hypergeometricsetup{symbol=G}\pFq[skip=4]{1}{1}{a}{0}{z} &= \dfrac{1}{\Gamma(a+1)}+O\!\left(z\log{z}\right), \nonumber \\\hypergeometricsetup{symbol=G}\pFq[skip=4]{1}{1}{a}{c}{z} &= \dfrac{\Gamma(1-c)}{\Gamma(a-c+1)}+O\!\left(z\right), \quad \Re(c) \le 0,\;c \neq 0. \nonumber
\end{align}
\end{custom}
\begin{custom}{Limiting Form at Infinity (\cite[\S 13.2(iv)]{NIST})} Except when $a=0,-1,\dots$, which corresponds to the polynomial cases,
\begin{equation}
\label{inf1}
\hypergeometricsetup{symbol=\mathbf{F}}\pFq[skip=4]{1}{1}{a}{c}{z} \sim\dfrac{e^zz^{a-c}}{\Gamma(a)}, \quad \left|\text{arg}(z)\right| \le \dfrac{1}{2}\pi-\delta, \quad z \to \infty.
\end{equation}
As for ${}_1G_1(a,c;z)$,
\begin{equation}
\label{inf2}
\hypergeometricsetup{symbol=G}\pFq[skip=4]{1}{1}{a}{c}{z} \sim z^{-a}, \quad \left|\text{arg}(z)\right| \le \dfrac{3}{2}\pi-\delta, \quad z \to \infty,
\end{equation}
without restrictions on the parameters. In both cases, $\delta$ is an arbitrarily small positive constant.
\end{custom}
\begin{custom}{Wronskians (\cite[\S 13.2(vi)]{NIST})}
\begin{align}
    W\!\left\{\hypergeometricsetup{symbol=\mathbf{F}}\pFq[skip=4]{1}{1}{a}{c}{z},z^{1-c}\hypergeometricsetup{symbol=\mathbf{F}}\pFq[skip=4]{1}{1}{a-c+1}{2-c}{z}\right\}\! &= \dfrac{\sin{(\pi c)}z^{-c}e^z}{\pi}, \label{wro1c}\\
    W\!\left\{\hypergeometricsetup{symbol=\mathbf{F}}\pFq[skip=4]{1}{1}{a}{c}{z},\hypergeometricsetup{symbol=\mathbf{G}}\pFq[skip=4]{1}{1}{a}{c}{z}\right\}\! &= \dfrac{-z^{-c}e^z}{\Gamma(a)},\label{wro2c}\\
    W\!\left\{\hypergeometricsetup{symbol=\mathbf{F}}\pFq[skip=4]{1}{1}{a}{c}{z},e^z\hypergeometricsetup{symbol=\mathbf{G}}\pFq[skip=4]{1}{1}{c-a}{c}{e^{\pm\pi i}z}\right\}\! &= \dfrac{e^{\mp\pi i}z^{-c}e^z}{\Gamma(c-a)}, \nonumber \\
    W\!\left\{z^{1-c}\hypergeometricsetup{symbol=\mathbf{F}}\pFq[skip=4]{1}{1}{a-c+1}{2-c}{z},\hypergeometricsetup{symbol=\mathbf{G}}\pFq[skip=4]{1}{1}{a}{c}{z}\right\}\! &= \dfrac{-z^{-c}e^z}{\Gamma(a-c+1)}, \nonumber \\
    W\!\left\{z^{1-c}\hypergeometricsetup{symbol=\mathbf{F}}\pFq[skip=4]{1}{1}{a-c+1}{2-c}{z},e^z\hypergeometricsetup{symbol=\mathbf{G}}\pFq[skip=4]{1}{1}{c-a}{c}{e^{\pm\pi i}z}\right\}\! &= \dfrac{-z^{-c}e^z}{\Gamma(1-a)}, \label{wro5c}\\
    W\!\left\{\hypergeometricsetup{symbol=\mathbf{G}}\pFq[skip=4]{1}{1}{a}{c}{z},e^z\hypergeometricsetup{symbol=\mathbf{G}}\pFq[skip=4]{1}{1}{c-a}{c}{e^{\pm\pi i}z}\right\}\! &= e^{\pm(a-c)\pi i}z^{-c}e^z. \nonumber
\end{align}
\end{custom}
\begin{custom}{Relations to Other Functions (\cite[\S 13.6(v)]{NIST})} For particular values of the parameters, the series associated to ${}_1G_1(a,c;z)$ and ${}_1F_1(a,c;z)$ are terminating and reduce to OPs. Some special cases are the following:
\begin{equation}
\label{specialCase1Confluent}
    \hypergeometricsetup{symbol=G}\pFq[skip=4]{1}{1}{-n}{\alpha+1}{z}=(-1)^n(\alpha+1)_n\hypergeometricsetup{symbol=F}\pFq[skip=4]{1}{1}{-n}{\alpha+1}{z}=(-1)^nn!L^{(\alpha)}_n(z).
\end{equation}
\end{custom}

\bibliographystyle{plain}
\bibliography{main} 
\section*{Acknowledgements}
Author wishes to express his gratitude to his advisor, Kenier Castillo, for suggesting this topic for the survey, the continuous encouragement and constructive guidance. This work is partially supported by the Centre for Mathematics of the University of Coimbra funded by the Portuguese Government through FCT/MCTES, DOI 10.54499/UIDB/00324/2020. 

Author is supported by FCT grant UI.BD.154694.2023.

 \end{document}